\documentclass{amsart} 
 
\usepackage{amsmath, amsthm, amssymb}
\usepackage{mathrsfs, bm, txfonts}
\usepackage{hyperref}
\usepackage[pdftex]{graphicx}
\usepackage{graphics}
\usepackage{caption}
\usepackage{subcaption}

\newcommand{\Rone}{\mathbb{R}}

\newcommand{\Rn}{\mathbb{R}^n}
\newcommand{\Rnone}{\mathbb{R}^{n+1}}

\newcommand{\Snm}{\mathbb{S}^{n-1}}

\newcommand{\CR}{\mathscr{C}_{R,d}^n}
\newcommand{\CRc}{\overline{\mathscr{C}}_{R,d}^n}

\newcommand{\Sdp}{\mathscr{S}_{\frac{d}{\pi}}^1}

\newcommand{\kapsurf}[1]{\bm{\kappa}_{#1}}
\newcommand{\kapbase}{\bm{\kappa}_{R}}

\newcommand{\musurf}[1]{\mu_{#1}}

\newcounter{desccount}
\newcommand{\descitem}[1]{%
  \item[#1] \refstepcounter{desccount}\label{#1}
}
\newcommand{\descref}[1]{\hyperref[#1]{#1}}

\newtheorem{theorem}{Theorem}[section]
\newtheorem{lemma}[theorem]{Lemma}

\newtheorem{corollary}[theorem]{Corollary}

\newtheorem{remark}[theorem]{Remark}

\title{Stability of near cylindrical stationary solutions to weighted-volume preserving curvature flows}

\author{David Hartley}

\address{}

\email{david.hartley@monash.edu}

\date{\today}

\begin{document}


\begin{abstract}
In this paper we consider a class of weighted-volume preserving curvature flows acting on hypersurfaces that are trapped within two parallel hyperplanes and satisfy an orthogonal boundary condition. In the author's thesis the stability of cylinders under the flows was considered; it was found that they are stable provided the radius satisfies a certain condition. Here we consider flows where there is a critical radius such that the cylinders are only stable on one side of the critical value, and the weight function is a linear combination of the elementary symmetric functions, the reason for such a choice is made clear in the appendix. We find that in such instances bifurcation from the cylindrical stationary solutions occurs at the critical radius and we determine a condition on the speed and weight functions such that the nearby non-cylindrical stationary solutions are stable under the flow. We will also highlight the specific cases of homogeneous speed functions and the mixed-volume preserving mean curvature flows.
\end{abstract}


\maketitle


\section{Introduction}
For a hypersurface $\Omega_0$ given by an embedding $\tilde{\bm{X}}_0:M^n\rightarrow\Rnone$, where $M^n$ is an $n$-dimensional manifold, its weighted-volume preserving curvature flow is a family of hypersurfaces $\left\{\Omega_t:t\in[0,T)\right\}$, $T>0$, such that $\Omega_t=\bm{X}\left(M^n,t\right)$ and $\bm{X}$ satisfies
\begin{equation}\label{WVPCF}
\frac{\partial\bm{X}}{\partial t}=\left(\frac{1}{\int_{M^n}\Xi\left(\bm{\kappa}\right)\,d\mu_t}\int_{M^n}F\left(\bm{\kappa}\right)\Xi\left(\bm{\kappa}\right)\,d\mu_t -F\left(\bm{\kappa}\right)\right)\bm{\nu},\ \ \bm{X}\left(\cdot,0\right)=\tilde{\bm{X}}_0,
\end{equation}
where $F\left(\bm{\kappa}\right)$ and $\Xi\left(\bm{\kappa}\right)$ are smooth, symmetric functions of the principal curvatures, $\bm{\kappa}=(\kappa_1,\ldots,\kappa_n)$, $\bm{\nu}$ is a choice of unit normal, and $\,d\mu_t$ is the volume form of the hypersurface $\Omega_t$. Short time existence for initially near cylindrical hypersurfaces, under assumptions \descref{(A1)}-\descref{(A3)} below, was proved in \cite{HartleyPhD}.

This flow is a generalisation of the volume preserving mean curvature flow (VPMCF), which has $F\left(\bm{\kappa}\right)=\sum_{a=1}^n\kappa_a$ and $\Xi(\bm{\kappa})=1$. In the case of compact, convex without boundary hypersurfaces, the VPMCF has been show to exist for all time and the hypersurfaces $\Omega_t$ converge to a sphere of the same enclosed volume as $\Omega_0$ as $t\rightarrow\infty$ \cite{Gage86,Huisken87}. The stability of spheres as stationary solutions to the VPMCF has previously been studied by Escher and Simonett. In \cite{Escher98A} Escher and Simonett consider graphs over spheres and prove that if the height function is small in the little-H\"older space $h^{1,\beta}$, then the VPMCF exists for all time and the hypersurfaces converge to a sphere. This result also proves the existence of non-convex hypersurfaces that converge to spheres. Generalisations of the VPMCF to speeds that are a power of an elementary symmetric function has been considered in \cite{Cabezas10}, where convergence to a sphere was obtained for convex hypersurfaces under an additional pinching assumption. When the weight function is an elementary symmetric function the flow is the mixed-volume preserving curvature flow, which has been studied by McCoy in \cite{McCoy04} for speed functions given by the mean curvature and in \cite{McCoy05} for a more general class of speed functions. In both cases it was proved that if the initial hypersurface is strictly convex it will converge to a sphere under the flow.

In the present article we consider the setting introduced by Athanassenas in \cite{Athanassenas97}. That is, we consider hypersurfaces, $\Omega$, with non-empty boundary embedded in the domain
\begin{equation*}
W=\left\{\bm{x}\in\Rnone:0<x_{n+1}<d\right\},\ d>0,
\end{equation*}
such that $\partial\Omega\subset\partial W$ and $\Omega$ meets $\partial W$ orthogonally. In the paper by Athanassenas it was shown that the volume, $Vol$, enclosed by $\Omega_0$ and $\partial W$ is preserved under the VPMCF and, assuming an axial symmetry, it was proved that if the initial hypersurface satisfies
\begin{equation}\label{MariaAssump}
\left|\Omega_0\right|:=\int_{M^n}\,d\mu_0\leq\frac{Vol}{d},
\end{equation}
then the VPMCF will exist for all time and the hypersurfaces will converge to a cylinder. In this case the assumption (\ref{MariaAssump}) was used to ensure that the hypersurfaces $\Omega_t$ do not touch the axis of rotation. In \cite{Hartley13} the author considered graphs over cylinders and proved that a cylinder of radius
\begin{equation*}
R>\frac{d\sqrt{n-1}}{\pi},
\end{equation*}
is stable under the VPMCF, in the same sense as in \cite{Escher98A}. In \cite{HartleyPhD} this result was extended to the flow (\ref{WVPCF}) subject to the condition:
\begin{equation}\label{RCond}
R>\frac{d}{\pi}\sqrt{\frac{(n-1)\frac{\partial F}{\partial\kappa_1}\left(\kapbase\right)}{\frac{\partial F}{\partial\kappa_n}\left(\kapbase\right)}},
\end{equation}
where $\kapbase=\left(\frac{1}{R},\ldots,\frac{1}{R},0\right)$.
\begin{theorem}[\cite{HartleyPhD}]\label{MainRCyl}
Let $\CR$ represent a cylinder of radius $R$ and length $d$. Assuming \descref{(A1)}-\descref{(A3)} are satisfied for $\tilde{R}=R$, there exists a neighbourhood of zero, $O_c\subset h^{2,\alpha}_{\frac{\partial}{\partial z}}\left(\CRc\right)$ with $0<\alpha<1$ (see Section \ref{SecRedEq}), such that if $\Omega_0$ is a graph over $\CR$ with height function in $O_c$, then there exists $T>0$ such that the flow (\ref{WVPCF}) with orthogonal boundary condition exists for $t\in[0,T)$. Furthermore, if (\ref{RCond}) is satisfied, then $T=\infty$ and the hypersurfaces converge exponentially fast to a cylinder as $t\rightarrow\infty$, with respect to the $h^{2,\alpha}\left(\CRc\right)$ topology.
\end{theorem}

It is important to note that in general both sides of (\ref{RCond}) depend on $R$, however when the speed function is homogeneous in $\bm{\kappa}$ this is not the case and (\ref{RCond}) is true for $R$ large enough. For the axially symmetric VPMCF in two dimensions the stability result was also obtained by LeCrone in \cite{LeCronePre}. In that paper LeCrone also showed that the cylinder of radius $\frac{d}{\pi}$ is part of a continuous family of constant mean curvature (CMC) unduloids, which satisfy the orthogonal boundary condition. He then proved that the unduloids in this family close to the cylinder are unstable stationary solutions of the VPMCF. The aim of this paper is to generalise this final result to the weighted-volume preserving curvature flows in any dimension.

For the main analysis we make some assumptions on the form of $F$ and $\Xi$:
\begin{description}
	\descitem{(A1)} $F$ and $\Xi$ are smooth, symmetric functions
	\descitem{(A2)} $\frac{\partial F}{\partial \kappa_a}\left(\kapsurf{\tilde{R}}\right)>0$  for every $a=1,\ldots,n$ and for some $\tilde{R}\in\Rone^+$
	\descitem{(A3)} $\Xi\left(\kapsurf{\tilde{R}}\right)>0$ for some $\tilde{R}\in\Rone^+$
	\descitem{(A4)} $\Xi(\bm{\kappa})=\sum_{a=0}^nc_aE_a(\bm{\kappa})$
	\descitem{(A5)} There exists $R_{crit}>0$ such that, in a neighbourhood of $R_{crit}$, (\ref{RCond}) is satisfied for $R>R_{crit}$ ($R<R_{crit}$) and strictly not satisfied for $R< R_{crit}$ ($R> R_{crit}$),
\end{description}
where
\begin{equation}\label{EleDef}
E_a\left(\bm{\kappa}\right)=\sum_{1\leq b_1<\ldots<b_a\leq n}\prod_{i=1}^a\kappa_{b_i},
\end{equation}
are the elementary symmetric functions. The first two assumptions ensure isotropy and local parabolicity respectively, while \descref{(A3)} ensures a valid flow. Assumption \descref{(A4)} ensures that a type of weighted-volume is preserved under the flow, see Appendix \ref{AppInv}. While the final assumption ensures a change in sign of the critical eigenvalue. Since \descref{(A5)} is not the most descriptive assumption, we note here that it is satisfied when $F$ is homogeneous and we will often consider the revised assumption:
\begin{description}
	\descitem{(A5)*} $F$ is homogeneous of degree $k$, i.e. $F(\alpha\bm{\kappa})=\alpha^kF(\bm{\kappa})$ for all $\alpha>0$.
\end{description}
We note that when this assumption is used, we have $R_{crit}=\frac{d}{\pi}\sqrt{\frac{(n-1)\frac{\partial F}{\partial\kappa_1}\left(\kapsurf{1}\right)}{\frac{\partial F}{\partial\kappa_n}\left(\kapsurf{1}\right)}}$ and if \descref{(A2)} is true for some $\tilde{R}>0$ then it is true for all $\tilde{R}>0$.

\begin{theorem}\label{MainThm}
For the flow (\ref{WVPCF}), with assumptions \descref{(A1)}-\descref{(A5)} satisfied for $\tilde{R}=R_{crit}$, the cylinder of radius $R_{crit}$ is part of continuous family of stationary solutions that satisfy the orthogonal boundary condition. Furthermore, if (\ref{FCond}) holds then the stationary solutions close to the cylinder are stable under axially symmetric, weighted-volume preserving perturbations.
\end{theorem}

In (\ref{FCond}) we have used the functions $F_{a}(\eta)=\frac{\partial F}{\partial\kappa_a}\left(\kapsurf{\frac{n-1}{\eta}}\right)$, $a=\{1,n\}$, $F_{nn}(\eta)=\frac{\partial^2F}{\partial\kappa_n^2}\left(\kapsurf{\frac{n-1}{\eta}}\right)$ and $F_{nnn}(\eta)=\frac{\partial^3F}{\partial\kappa_n^3}\left(\kapsurf{\frac{n-1}{\eta}}\right)$ to simplify the notation. Also note that when \descref{(A5)*} is satisfied then these functions are homogeneous, so have the representations $F_a(\eta)=\eta^{k-1}F_a$, $F_{nn}(\eta)=\eta^{k-2}F_{nn}$ and $F_{nnn}(\eta)=\eta^{k-3}F_{nnn}$, where we define the constants $F_a:=F_a(1)$, $F_{nn}:=F_{nn}(1)$ and $F_{nnn}:=F_{nnn}(1)$.

\begin{remark}\label{ThmRem}
For $n=4$, an example of a non-homogeneous speed function that satisfies assumption \descref{(A1)}, \descref{(A2)} (for all $R$) and \descref{(A5)} is
		\begin{equation*}
			F\left(\bm{\kappa}\right)=\kappa_1\kappa_2\kappa_3\kappa_4 +\frac{\pi^2}{12d^2}\left(\kappa_1^2+\kappa_2^2+\kappa_3^2+\kappa_4^2\right) +\frac{\pi^2}{18d^2}\left(\kappa_1^3+\kappa_2^3+\kappa_3^3+\kappa_4^3\right),
		\end{equation*}
		in which case $R_{crit}=1$. However if we change the speed function slightly to
		\begin{equation*}
			F\left(\bm{\kappa}\right)=\kappa_1\kappa_2\kappa_3\kappa_4 +\frac{\pi^2}{6d^2}\left(\kappa_1^2+\kappa_2^2+\kappa_3^2+\kappa_4^2\right) +\frac{\pi^2}{18d^2}\left(\kappa_1^3+\kappa_2^3+\kappa_3^3+\kappa_4^3\right),
		\end{equation*}
		then (\ref{RCond}) is never satisfied so assumption \descref{(A5)} is false (while both \descref{(A1)} and \descref{(A2)} are still satisfied for all $\tilde{R}$). This makes the classification of allowable speed functions difficult.
\end{remark}

\begin{corollary}\label{HomogCor}
For the flow (\ref{WVPCF}), with assumptions \descref{(A1)}- \descref{(A5)*} satisfied for $\tilde{R}=\frac{d}{m\pi}\sqrt{\frac{(n-1)F_1}{F_n}}$ with $m\in\mathbb{N}$, then the cylinder of radius $\frac{d}{m\pi}\sqrt{\frac{(n-1)F_1}{F_n}}$ is part of continuous family of stationary solutions that satisfy the orthogonal boundary condition. Furthermore, if $m=1$ and
\begin{align}\label{HomogCond}
&\frac{-6(n-1)\sum_{a=1}^n\frac{c_a\pi^a}{d^a}\left(\frac{F_n}{(n-1)F_1}\right)^{\frac{a}{2}}\left(\binom{n-2}{a-1} -\frac{F_1}{F_n}\binom{n-1}{a-1}\right)}{\sum_{a=0}^n\frac{c_a\pi^a}{d^a}\left(\frac{F_n}{(n-1)F_1}\right)^{\frac{a}{2}}\binom{n-1}{a}} -k^2+6n+6-\frac{3F_1^2F_{nnn}}{2F_n^3}\\
&\hspace{0.6cm}  +\frac{2F_1^2F_{nn}^2}{F_n^4} +\frac{kF_1F_{nn}}{2F_n} +\frac{(n-1)F_1^2F_{nn}}{F_n^3}-\frac{(n-1)^2F_1^2}{F_n^2} -\frac{(n-1)(k-3)F_1}{F_n}<0,\nonumber
\end{align}
holds, then the stationary solutions close to the cylinder of radius $\frac{d}{\pi}\sqrt{\frac{(n-1)F_1}{F_n}}$ are stable under axially symmetric, weighted-volume preserving perturbations.
\end{corollary}

Note that the first result of this Corollary is slightly stronger than what is immediately obtained from Theorem \ref{MainThm} as it proves the existence of a sequence of bifurcation points, see Corollary \ref{StatSolHomog}. However, the remainder of Corollary \ref{HomogCor} follows straight from Theorem \ref{MainThm} by using the homogeneous representations of $F_a(\eta)$, $F_{nn}(\eta)$ and $F_{nnn}(\eta)$.

\begin{corollary}\label{MainCor}
The cylinder of radius $\frac{d\sqrt{n-1}}{\pi}$ is part of a continuous family of CMC unduloids that satisfy the orthogonal boundary condition. Furthermore, under the VPMCF along with orthogonal boundary condition, the unduloids close to the cylinder are unstable stationary solutions in dimensions $2\leq n\leq10$, while for $n\geq11$ they are stable under axially symmetric, volume preserving perturbations.
\end{corollary}

\begin{remark}\label{MainRem}
\begin{enumerate}
	\item The first statement of the corollary is easily seen from the work of Delaunay \cite{Delaunay41} when $n=2$ and Hsiang and Yu \cite{Hsiang81} in higher dimensions. It was also shown in \cite{Hsiang81} that this family of CMC unduloids limits to spheres of radius $d$.
	\item It can also be seen from Corollary \ref{HomogCor} that the cylinders of radius $\frac{md\sqrt{n-1}}{\pi}$, $m\in\mathbb{N}$ are also part of a continuous family of CMC unduloids that satisfy the orthogonal boundary condition. However these cylinders, and nearby unduloids, are known to be linearly unstable under the flows.
	\item The stability of unduloids under VPMCF in high dimensions is an interesting result and it is not currently known why such a change in stability occurs at $n=11$.
	\item Note that Corollary \ref{MainCor} does not contradict the theorem in \cite{Athanassenas97} since for a critical cylinder $\frac{Vol}{d\left|\Omega\right|}=\frac{\sqrt{n-1}}{n\pi}<\frac{1}{6}$, so that (\ref{MariaAssump}) is not satisfied and will not be satisfied for unduloids close to this cylinder, or the hypersurfaces close to them.
\end{enumerate}
\end{remark}

In Section \ref{SecRedEq} we introduce the axially symmetric flow and remove the boundary conditions by considering it as an even function on the circle of radius $\frac{d}{\pi}$. We then introduce an invariant parameter, based on the weighted-volume, into the flow that enables us to reduce the problem to one on even, mean zero functions on the circle. Section \ref{SecBifAnal} analyses the reduced equation and proves Theorem \ref{MainThm}, while Section \ref{SecMVPMCF} considers the specific case of mixed-volume preserving mean curvature flow and in doing so proves Corollary \ref{MainCor}. In Section \ref{SecGeoCons} we provide an alternate construction of the family of CMC unduloids, using their representation from \cite{Hsiang81}. This is then used to provide more insight into the results of Section \ref{SecMVPMCF}.

The author is thankful to the University of Queensland and Dr Artem Pulemotov along with Monash University and Dr Maria Athanassenas for their support while preparing this paper. Part of this work was initiated during the author's PhD candidature at Monash University.


\section{Reducing the Equation}\label{SecRedEq}
We consider normal graphs over cylinders, hence $M^n=\Snm\times(0,d)$, and embeddings that are axially symmetric:
\begin{equation*}
\bm{X}_{\rho}\left(\bm{\theta},z\right)=\left(\rho(z)Y^{n-1}_{1,1}\left(\bm{\theta}\right),\ldots,\rho(z)Y^{n-1}_{1,n}\left(\bm{\theta}\right),z\right)
\end{equation*}
where $\rho:[0,d]\rightarrow\Rone^+$ and $Y^{n-1}_{1,a}:\Snm\rightarrow\Rone$, $1\leq a\leq n$, are the first order spherical harmonics on $\Snm$. We also restrict to functions so that the orthogonal boundary condition is satisfied, i.e. $\left.\frac{d\rho}{dz}\right|_{z=0,d}=0$. Up to a tangential diffeomorphism the flow (\ref{WVPCF}) is equivalent to a PDE on the function $\rho:[0,d]\times[0,T)\rightarrow\Rone^+$:
\begin{equation}\label{WVPCFGraph}
\begin{array}{c}\frac{\partial \rho}{\partial t}=\sqrt{1+\left(\frac{\partial\rho}{\partial z}\right)^2}\left(\frac{1}{\int_{[0,d]}\Xi\left(\kapsurf{\rho}\right)\,d\musurf{\rho}}\int_{[0,d]}F\left(\kapsurf{\rho}\right) \Xi\left(\kapsurf{\rho}\right)\,d\musurf{\rho}-F\left(\kapsurf{\rho}\right)\right),\\ \left.\frac{\partial\rho}{\partial z}\right|_{z=0,d}=0,\hspace{0.7cm}\rho(\cdot,0)=\rho_0,\end{array}
\end{equation}
where $\tilde{\bm{X}}_{0}=\bm{X}_{\rho_0}$, $\,d\musurf{\rho}=\rho^{n-1}\sqrt{1+\left(\frac{\partial\rho}{\partial z}\right)^2}\,dz$ and $\kapsurf{\rho}$ is the principal curvature vector of the embedding $\bm{X}_{\rho}$.

Throughout this paper we will be considering functions in the little-H\"older spaces, which are defined for an open set $U\subset\Rn$, $k\in\mathbb{N}$ and $\alpha\in(0,1)$ by
\begin{align*}
&h^{0,\alpha}\left(\bar{U}\right)=\left\{f\in C^{0,\alpha}\left(\bar{U}\right):\lim_{r\rightarrow0}\sup_{\stackrel{x,y\in\bar{U}}{0<|x-y|<r}}\frac{|f(x)-f(y)|}{|x-y|^{\alpha}}=0\right\},\\
h^{k,\alpha}\left(\bar{U}\right)=&\left\{f\in C^{k,\alpha}\left(\bar{U}\right):D^{\beta}f\in h^{0,\alpha}\left(\bar{U}\right)\text{ for all multi-indices }\beta\text{ with }|\beta|=k\right\},
\end{align*}
where $C^{k,\alpha}\left(\bar{U}\right)$ are the H\"older spaces. These spaces are equipped with the same norm as the H\"older spaces and can be extended to a manifold by means of an atlas. Importantly these spaces interpolate between themselves when using the continuous interpolation functor $(\cdot,\cdot)_{\theta}$, $\theta\in(0,1)$, \cite{Lunardi95}. That is, for $k,l\in\mathbb{N}$, $\alpha,\beta\in(0,1)$ with $k+\alpha>l+\beta$ then
\begin{equation}\label{lHInterp}
\left(h^{l,\beta},h^{k,\alpha}\right)_{\theta}=h^{\theta k+(1-\theta)l+\theta\alpha-(1-\theta)\beta}
\end{equation}
for $\theta\in(0,1)$ provided $\theta k+(1-\theta)l+\theta\alpha-(1-\theta)\beta\notin\mathbb{Z}$, see \cite{Guenther02}. Note that here we use the notation $h^{\sigma}=h^{\lfloor\sigma\rfloor,\sigma-\lfloor\sigma\rfloor}$ for $\sigma\in\Rone$. To restrict to functions that satisfy our boundary condition we define:
\begin{equation*}
h^{k,\alpha}_{B}\left(\overline{M}^n\right)=\left\{f\in h^{k,\alpha}\left(\overline{M}^n\right):\left.B[f]\right|_{\partial M^n}=0\right\},
\end{equation*}
for $k\in\mathbb{N}_0$ and $\alpha\in(0,1)$.

To remove the boundary condition we consider $h^{k,\alpha}$ functions on the circle of radius $\frac{d}{\pi}$, $\Sdp$, with an even symmetry; we denote this space of functions by $h^{k,\alpha}_e\left(\Sdp\right)$. Also, denote the even extension of a function $\rho$ from $[0,d]$ to $\Sdp$ by $u_{\rho}$, i.e.
\begin{equation*}
u_{\rho}(z)=\left\{\begin{array}{ll} \rho(z) & z\in[0,d]\\ \rho(-z) & z\in(-d,0)\end{array}\right..
\end{equation*}
Lastly, for an even function $u$, define $\left.u\right|_{[0,d]}$ to be the function on $[0,d]$ such that its even extension to the circle is $u$. Using this notation we see that solving (\ref{WVPCFGraph}) is then equivalent to solving
\begin{equation}\label{WVPCFCircle}
\frac{\partial u}{\partial t}=\sqrt{1+u'^2}\left(\frac{1}{\int_{\Sdp}\Xi\left(\kapsurf{u}\right)\,d\musurf{u}}\int_{\Sdp}F\left(\kapsurf{u}\right)\Xi\left(\kapsurf{u}\right)\,d\musurf{u}-F\left(\kapsurf{u}\right)\right),\hspace{0.5cm}u(\cdot,0)=u_{\rho_0},
\end{equation}
where $z$ is an arc-length variable, $\,d\musurf{u}=u^{n-1}\sqrt{1+u'^2}\,dz$ and $\kapsurf{u_{\rho}}$ is the even extension of $\kapsurf{\rho}$, i.e.
\begin{equation*}
\kapsurf{u}=\left(\frac{1}{u\sqrt{1+u'^2}},\ldots,\frac{1}{u\sqrt{1+u'^2}},-\frac{u''}{\left(1+u'^2\right)^{\frac{3}{2}}}\right),
\end{equation*}
where we use $'$ to denote derivatives with respect to $z$. The equivalence is easily seen because the right hand side of (\ref{WVPCFCircle}) is an even function when $u$ is even.

We now consider the function $Q$:
\begin{equation}\label{defQ}
Q(u):=c_0u^n+\left(\sum_{a=1}^{n}\frac{nc_a}{a}E_{a-1}\left(\kapsurf{u}\right)\right)u^{n-1}\sqrt{1+u'^2},
\end{equation}
and define
\begin{equation*}
\tilde{Q}(\eta):=Q\left(\frac{n-1}{\eta}\right)=\sum_{a=0}^{n}c_a(n-1)^{n-a}\binom{n}{a}\eta^{-(n-a)}.
\end{equation*}
Note that $\tilde{Q}'(\eta)=-\frac{n(n-1)^n}{\eta^{n+1}}\Xi\left(\kapsurf{\frac{n-1}{\eta}}\right)$ so the function is decreasing at $\eta$ whenever \descref{(A3)} holds for $\tilde{R}=\frac{n-1}{\eta}$, in particular we have a local inverse of $\tilde{Q}$ at these points. Corollary \ref{CorInvar} gives us that the weighted-volume
\begin{align*}
WVol(u)=&\int_{\Sdp}Q(u)\,dz\\
=&c_0\int_{\Sdp}u^n\,dz+\sum_{a=1}^n\frac{nc_a}{a}\int_{\Sdp}E_{a-1}\left(\kapsurf{u}\right)\,d\musurf{u}\\
=&\frac{2}{\omega_n}\left(c_0 V_{n+1}(u|_{[0,d]})+\sum_{a=1}^nc_a\binom{n+1}{a}V_{n+1-a}(u|_{[0,d]})\right)
\end{align*}
is an invariant of the flow. Here  we use $\omega_n$ to denote the volume of a unit $n$-ball and $V_b(\rho)$ to denote the mixed-volume of the hypersurface defined by $\bm{X}_{\rho}$. The mixed-volumes are usually only defined for convex hypersurfaces, however they can be extended to all hypersurfaces by using the formula:
\begin{equation}\label{MixedV}
V_{b}=\left\{\begin{array}{ll}\frac{1}{(n+1)\binom{n}{n-b}}\int_{M^n}E_{n-b}\left(\bm{\kappa}\right)\,d\mu & b=1,\ldots,n,\\ Vol & b=n+1.\end{array}\right.
\end{equation}
Note that in the case $c_a=\delta_{ab}$ for some $0\leq b\leq n-1$, $WVol(u)$ is (up to a positive multiple) the $(n+1-b)$\textsuperscript{th} mixed volume $V_{n+1-b}(u|_{[0,d]})$, the $b=n$ case is excluded as then \descref{(A3)} is not satisfied for any $\tilde{R}$.

We now follow \cite{LeCronePre} in introducing the function $\psi_{\eta_0}$ from an open set of $h^{k,\alpha}_{e,0}\left(\Sdp\right)\times\Rone^+$ to $h^{k,\alpha}_e\left(\Sdp\right)$, where
\begin{equation*}
h^{k,\alpha}_{e,0}\left(\Sdp\right)=\left\{f\in h^{k,\alpha}_{e}\left(\Sdp\right):\int_{\Sdp}f\,dz=0\right\}.
\end{equation*}
To simplify the notation we define the projection $P_0:h^{0,\alpha}_e\left(\Sdp\right)\rightarrow h^{0,\alpha}_{e,0}\left(\Sdp\right)$:
\begin{equation*}
P_0[u]=u-\fint_{\Sdp}u\,dz,
\end{equation*}
and the operators $L(u):=\sqrt{1+u'^2}$ and
\begin{equation*}
G(u)=L(u)\left(\frac{1}{\int_{\Sdp}\Xi\left(\kapsurf{u}\right)\,d\musurf{u}}\int_{\Sdp}F\left(\kapsurf{u}\right)\Xi\left(\kapsurf{u}\right)\,d\musurf{u} -F\left(\kapsurf{u}\right)\right).
\end{equation*}

\begin{lemma}\label{defpsi}
For each $\eta_0\in\Rone^+$ such that $\Xi\left(\kapsurf{\frac{n-1}{\eta_0}}\right)>0$ there exist $U_{\eta_0}$, $V_{\eta_0}$, neighbourhoods of $(0,\eta_0)\in h_{e,0}^{2,\alpha}\left(\Sdp\right)\times\Rone^+$ and $\frac{n-1}{\eta_0}\in h_e^{2,\alpha}\left(\Sdp\right)$ respectively, and a smooth diffeomorphism $\psi_{\eta_0}:U_{\eta_0}\rightarrow V_{\eta_0}$ such that:
\begin{itemize}
	\item $P_0\left[\psi_{\eta_0}\left(\bar{u},\eta\right)\right]=\bar{u}$ for all $\left(\bar{u},\eta\right)\in U_{\eta_0}$,
	\item $\fint_{\Sdp}Q\left(\psi_{\eta_0}\left(\bar{u},\eta\right)\right)\,dz=\tilde{Q}\left(\eta\right)$ for all $\left(\bar{u},\eta\right)\in U_{\eta_0}$,
	\item For $u\in V_{\eta_0}$, $\psi_{\eta_0}^{-1}(u)=\left(P_0[u],\tilde{Q}^{-1}\left(\fint_{\Sdp}Q(u)\,dz\right)\right)$. In particular $\psi_{\eta_0}\left(0,\eta\right)=\frac{n-1}{\eta}$ for all $\left(0,\eta\right)\in U_{\eta_0}$,
	\item $\Xi\left(\kapsurf{\psi_{\eta_0}(\bar{u},\eta)}\right)>0$ for all $\left(\bar{u},\eta\right)\in U_{\eta_0}$.
\end{itemize}
Furthermore:
\begin{equation}\label{Dpsi}
D_1\psi_{\eta_0}\left(\bar{u},\eta\right)[\bar{v}]=\bar{v}-\frac{\int_{\Sdp}\bar{v}\Xi\left(\kapsurf{\psi_{\eta_0}(\bar{u},\eta)}\right)\psi_{\eta_0}\left(\bar{u},\eta\right)^{n-1}\,dz}{\int_{\Sdp}\Xi\left(\kapsurf{\psi_{\eta_0}(\bar{u},\eta)}\right)\psi_{\eta_0}\left(\bar{u},\eta\right)^{n-1}\,dz}.
\end{equation}
\end{lemma}

\begin{proof}
We consider the operator
\begin{equation*}
\Phi\left(u,\bar{u},\eta\right)=\left(P_0[u]-\bar{u},\fint_{\Sdp}Q(u)\,dz-\tilde{Q}\left(\eta\right)\right),
\end{equation*}
which is a generalisation of the one in \cite{LeCronePre}, and note that $\Phi\left(\frac{n-1}{\eta_0},0,\eta_0\right)=(0,0)$. We linearise with respect to the $u$ function around the point $\left(\frac{n-1}{\eta_0},0,\eta_0\right)$ and act on $v\in h_e^{2,\alpha}\left([0,d]\right)$:
\begin{align*}
D_1\Phi\left(\frac{n-1}{\eta_0},0,\eta_0\right)[v]=&\left(P_0[v],\fint_{\Sdp}DQ\left(\frac{n-1}{\eta_0}\right)[v]\,dz\right)\\
=&\left(P_0[v],n\fint_{\Sdp}v\Xi\left(\kapsurf{\frac{n-1}{\eta_0}}\right)\left(\frac{n-1}{\eta_0}\right)^{n-1}\,dz\right),
\end{align*}
where we have used (\ref{QIdent}). Since $\Xi\left(\kapsurf{\frac{n-1}{\eta_0}}\right)\neq0$ this operator has trivial kernel. Therefore by the implicit function theorem we obtain the existence of $\psi_{\eta_0}$ locally around $(0,\eta_0)\in h_{e,0}^{2,\alpha}\left(\Sdp\right)\times\Rone^+$ such that it satisfies the first three dot points and, if necessary, we shrink $U_{\eta_0}$ to ensure the fourth holds.

To obtain the linearisation of $\psi_{\eta_0}$ we consider
\begin{equation}\label{defphi}
\phi(\bar{u},\eta)=\Phi\left(\psi_{\eta_0}\left(\bar{u},\eta\right),\bar{u},\eta\right)=\left(P_0\left[\psi_{\eta_0}\left(\bar{u},\eta\right)\right]-\bar{u},\fint_{\Sdp}Q\left(\psi_{\eta_0}\left(\bar{u},\eta\right)\right)\,dz-\tilde{Q}\left(\eta\right)\right).
\end{equation}
By linearising with respect to $\bar{u}$ we obtain:
\begin{align*}
D_1\phi(\bar{u},\eta)[\bar{v}]=&\left(P_0\left[D_1\psi_{\eta_0}(\bar{u},\eta)[\bar{v}]\right]-\bar{v},\fint_{\Sdp}DQ\left(\psi_{\eta_0}(\bar{u},\eta)\right)\left[D_1\psi_{\eta_0}(\bar{u},\eta)[\bar{v}]\right]\,dz\right)\\
=&\left(P_0\left[D_1\psi_{\eta_0}(\bar{u},\eta)[\bar{v}]\right]-\bar{v},n\fint_{\Sdp}D_1\psi_{\eta_0}(\bar{u},\eta)[\bar{v}]\Xi\left(\kapsurf{\psi_{\eta_0}(\bar{u},\eta)}\right)\psi_{\eta_0}(\bar{u},\eta)^{n-1}\,dz\right),
\end{align*}
where we have used (\ref{QIdent}) again. By the properties of $\psi_{\eta_0}$ we have that $\phi(\bar{u},\eta)=(0,0)$ for all $\left(\bar{u},\eta\right)\in U_{\eta_0}$, hence we obtain:
\begin{equation*}
\left(P_0\left[D_1\psi_{\eta_0}(\bar{u},\eta)[\bar{v}]\right]-\bar{v},n\fint_{\Sdp}D_1\psi_{\eta_0}(\bar{u},\eta)[\bar{v}]\Xi\left(\kapsurf{\psi_{\eta_0}(\bar{u},\eta)}\right)\psi_{\eta_0}(\bar{u},\eta)^{n-1}\,dz\right)=(0,0).
\end{equation*}
From the first of these equations we obtain $D_1\psi_{\eta_0}(\bar{u},\eta)[\bar{v}]=\bar{v}+\fint_{\Sdp}D_1\psi_{\eta_0}(\bar{u},\eta)[\bar{v}]\,dz$, so by substituting this into the second equation we obtain:
\begin{equation*}
\fint_{\Sdp}\bar{v}\Xi\left(\kapsurf{\psi_{\eta_0}(\bar{u},\eta)}\right)\psi_{\eta_0}(\bar{u},\eta)^{n-1}\,dz +\fint_{\Sdp}\Xi\left(\kapsurf{\psi_{\eta_0}(\bar{u},\eta)}\right)\psi_{\eta_0}(\bar{u},\eta)^{n-1}\,dz\fint_{\Sdp}D_1\psi_{\eta_0}(\bar{u},\eta)[\bar{v}]\,dz=0.
\end{equation*}
Therefore
\begin{equation*}
\fint_{\Sdp}D_1\psi_{\eta_0}(\bar{u},\eta)[\bar{v}]\,dz=-\frac{\fint_{\Sdp}\bar{v}\Xi\left(\kapsurf{\psi_{\eta_0}(\bar{u},\eta)}\right)\psi_{\eta_0}(\bar{u},\eta)^{n-1}\,dz}{\fint_{\Sdp}\Xi\left(\kapsurf{\psi_{\eta_0}(\bar{u},\eta)}\right)\psi_{\eta_0}(\bar{u},\eta)^{n-1}\,dz},
\end{equation*}
and the result follows by substituting this into $D_1\psi_{\eta_0}(\bar{u},\eta)[\bar{v}]=\bar{v}+\fint_{\Sdp}D_1\psi_{\eta_0}(\bar{u},\eta)[\bar{v}]\,dz$.
\end{proof}

We also note that from (\ref{defphi}) and $\phi(\bar{u},\eta)=(0,0)$ we have the representation $\psi_{\eta_0}(\bar{u},\eta)=\bar{u}+\fint_{\Sdp}\psi_{\eta_0}(\bar{u},\eta)\,dz$; in particular this means that $\frac{d\psi_{\eta_0}(\bar{u},\eta)}{dz}=\frac{d\bar{u}}{dz}$. We are now able to introduce an equivalent equation on the reduced space $h_{e,0}^{2,\alpha}\left(\Sdp\right)$.
\begin{corollary}\label{EquivEq}
Let $\eta_0\in\Rone^+$ be such that \descref{(A1)}-\descref{(A4)} are satisfied with $\tilde{R}=\frac{n-1}{\eta_0}$, then the flow
\begin{equation}\label{RedWVPCF}
\frac{\partial\bar{u}}{\partial t}=\bar{G}_{\eta_0}(\bar{u},\eta):=P_0\left[G\left(\psi_{\eta_0}(\bar{u},\eta)\right)\right],\ \ \bar{u}(\cdot,0)=\bar{u}_0,
\end{equation}
for some fixed $\eta$ such that $(\bar{u}_0,\eta)\in U_{\eta_0}$, is equivalent to the flow (\ref{WVPCFCircle}).

That is, if $\bar{u}$ is a solution to (\ref{RedWVPCF}) with initial condition $\bar{u}_0$ then $\psi_{\eta_0}(\bar{u},\eta)$ is a solution to (\ref{WVPCFCircle}) with initial condition $\psi_{\eta_0}(\bar{u}_0,\eta)$. Conversely, if $u$ is a solution to (\ref{WVPCFCircle}) with initial condition $u_0$ such that $u(t)\in V_{\eta_0}$ for each $t$, then $P_0[u]$ is a solution to (\ref{RedWVPCF}) with $\eta=\tilde{Q}^{-1}\left(\fint_{\Sdp}Q(u_0)\,dz\right)$.
\end{corollary}

\begin{proof}
Firstly let $u(t)$ be a solution to (\ref{WVPCFCircle}) in $V_{\eta_0}$ for all $t\in[0,\delta)$. Set $\eta=\tilde{Q}^{-1}\left(\fint_{\Sdp}Q(u_0)\,dz\right)$ and since $\fint_{\Sdp}Q(u)\,dz$ is an invariant of the flow we have
\begin{equation*}
\eta=\tilde{Q}^{-1}\left(\fint_{\Sdp}Q(u(t))\,dz\right),
\end{equation*}
for all $t\in[0,\delta)$. So by Lemma \ref{defpsi} we have that $u(t)=\psi_{\eta_0}\left(P_0[u(t)],\eta\right)$ for all $t\in[0,\delta)$. Substituting this into (\ref{WVPCFCircle}) gives
\begin{equation*}
\frac{\partial u}{\partial t}=G\left(\psi_{\eta_0}\left(P_0[u(t)],\eta\right)\right),\ t\in[0,\delta),
\end{equation*}
and by taking the projection of this we find that $P_0[u]$ is a solution to (\ref{RedWVPCF}), with parameter $\eta$ and initial condition $\bar{u}_0=P_0[u_0]$, for all $t\in[0,\delta)$.

Now let $\bar{u}(t)$, for $t\in[0,\delta)$, be a solution to (\ref{RedWVPCF}) with parameter $\eta$ and initial condition $\bar{u}_0$. We let $u(t)=\psi_{\eta_0}\left(\bar{u}(t),\eta\right)$ and compute it's derivative using (\ref{Dpsi}):
\begin{align*}
\frac{\partial u}{\partial t}=&D_1\psi_{\eta_0}\left(\bar{u},\eta\right)\left[\frac{\partial\bar{u}}{\partial t}\right]\\
=&P_0\left[G(u)\right] -\frac{\int_{\Sdp}P_0\left[G(u)\right]\Xi\left(\kapsurf{u}\right)u^{n-1}\,dz}{\int_{\Sdp}\Xi\left(\kapsurf{u}\right)u^{n-1}\,dz}\\
=&G(u) -\frac{\int_{\Sdp}G(u)\Xi\left(\kapsurf{u}\right)u^{n-1}\,dz}{\int_{\Sdp}\Xi\left(\kapsurf{u}\right)u^{n-1}\,dz}\\
=&G(u) -\frac{1}{\int_{\Sdp}\Xi\left(\kapsurf{u}\right)u^{n-1}\,dz}\int_{\Sdp}\left(\frac{\int_{\Sdp}F\left(\kapsurf{u}\right) \Xi\left(\kapsurf{u}\right)\,d\musurf{u}}{\int_{\Sdp}\Xi\left(\kapsurf{u}\right)\,d\musurf{u}} -F\left(\kapsurf{u}\right)\right)\Xi\left(\kapsurf{u}\right)\,d\musurf{u}\\
=&G(u).
\end{align*}
Therefore $u(t)$ solves (\ref{WVPCFCircle}) for $t\in[0,\delta)$ with initial condition $u_0=\psi_{\eta_0}\left(\bar{u}_0,\eta\right)$.
\end{proof}


\section{Non-trivial Stationary Solutions}\label{SecBifAnal}
The aim of this section is to prove Theorem \ref{MainThm}. We do it in two parts. We will first prove the existence of the stationary solutions, then determine the criteria for their stability. We will perform the analysis on the reduced equation (\ref{RedWVPCF}) then transfer the results to the full flow in Corollaries \ref{StatSol} and \ref{GenStable}, which together give Theorem \ref{MainThm}. We now fix $\eta_0=\frac{n-1}{R_{crit}}$.

\begin{theorem}\label{BifPoints}
Assume \descref{(A1)}-\descref{(A5)} are satisfied for $\tilde{R}=R_{crit}$. Then $(0,\eta_0)$ is a bifurcation point of $\bar{G}_{\eta_0}\left(\bar{u},\eta\right)=0$.

That is, there exists a curve $\left\{\left(\tilde{u}(s),\eta(s)\right):|s|<\delta\right\}\subset U_{\eta_0}$, $\delta>0$, such that $\left(\tilde{u}(0),\eta(0)\right)=(0,\eta_0)$, $\tilde{u}(s)\neq0$ for $s\neq0$, and $\bar{G}_{\eta_0}\left(\tilde{u}(s),\eta(s)\right)=0$ for all $|s|<\delta$. These are also the only non-trivial solutions to $\bar{G}_{\eta_0}\left(\bar{u},\eta\right)=0$ in a neighbourhood of $(0,\eta_0)$.
\end{theorem}

\begin{proof}
We calculate the linearisation of $\bar{G}_{\eta_0}(\bar{u},\eta)=P_0\left[G\left(\psi_{\eta_0}(\bar{u},\eta)\right)\right]$ with respect to $\bar{u}$ and act on $\bar{v}\in h_{e,0}^{2,\alpha}\left([0,d]\right)$:
\begin{equation}\label{D1Gbar}
D_1\bar{G}_{\eta_0}(\bar{u},\eta)[\bar{v}]=P_0\left[DG\left(\psi_{\eta_0}(\bar{u},\eta)\right)\left[D_1\psi_{\eta_0}(\bar{u},\eta)[\bar{v}]\right]\right].
\end{equation}

To simplify the calculation of $DG(u)$ we set $W(u)=\frac{F\left(\kapsurf{u}\right)\Xi\left(\kapsurf{u}\right)u^{n-1}L(u)}{\int_{\Sdp}\Xi\left(\kapsurf{u}\right)u^{n-1}L(u)\,dz}$ and use that $DL(u)[v]=\frac{u'v'}{L(u)}$:
\begin{align}\label{DG}
DG(u)[v]=&\frac{u'v'}{L(u)}\left(\int_{\Sdp}W(u)\,dz-F\left(\kapsurf{u}\right)\right)+L(u)\left(\int_{\Sdp}DW(u)[v]\,dz-D\left(F\left(\kapsurf{u}\right)\right)[v]\right)\nonumber\\
=&\frac{u'v'G(u)}{L(u)^2}+L(u)\left(\int_{\Sdp}DW(u)[v]\,dz-D\left(F\left(\kapsurf{u}\right)\right)[v]\right).
\end{align}

If we linearise $W$ and evaluate at $u=\frac{n-1}{\eta}$ we obtain
\begin{equation*}
\int_{\Sdp}DW\left(\frac{n-1}{\eta}\right)[v]\,dz=\fint_{\Sdp}\left.D\left(F\left(\kapsurf{u}\right)\right)\right|_{u=\frac{n-1}{\eta}}[v]\,dz.
\end{equation*}
Substituting this into (\ref{DG}):
\begin{equation}\label{DG0}
DG\left(\frac{n-1}{\eta}\right)[v]=\fint_{\Sdp}\left.D\left(F\left(\kapsurf{u}\right)\right)\right|_{u=\frac{n-1}{\eta}}[v]\,dz -\left.D\left(F\left(\kapsurf{u}\right)\right)\right|_{u=\frac{n-1}{\eta}}[v],
\end{equation}
Therefore using (\ref{Dpsi}) and (\ref{D1Gbar}) the linearisation of $\bar{G}_{\eta_0}$ at the point $(0,\eta)$ is
\begin{align}\label{D1Gbar0}
D_1\bar{G}_{\eta_0}(0,\eta)[\bar{v}]=&DG\left(\frac{n-1}{\eta}\right)[\bar{v}]\\
=&\fint_{\Sdp}\left.D\left(F\left(\kapsurf{u}\right)\right)\right|_{u=\frac{n-1}{\eta}}[\bar{v}]\,dz -\left.D\left(F\left(\kapsurf{u}\right)\right)\right|_{u=\frac{n-1}{\eta}}[\bar{v}]\nonumber.
\end{align}

To calculate the linearisation of $F\left(\kapsurf{u}\right)$ we use its symmetries:
\begin{equation}\label{DF}
D\left(F\left(\kapsurf{u}\right)\right)[v]=(n-1)\frac{\partial F}{\partial\kappa_1}\left(\kapsurf{u}\right)D\kappa_1(u)[v] +\frac{\partial F}{\partial\kappa_n}\left(\kapsurf{u}\right)D\kappa_n(u)[v],
\end{equation}
where
\begin{equation*}
\kappa_1(u)=\frac{1}{uL(u)},\ \ \ \kappa_n(u)=-\frac{u''}{L(u)^3}.
\end{equation*}
Taking the linearisation of the curvatures:
\begin{equation}\label{Dkap}
D\kappa_1(u)[v]=\frac{-v}{u^2L(u)}-\frac{u'v'}{uL(u)^3},\ \ \ D\kappa_n(u)[v]=-\frac{v''}{L(u)^3}+\frac{3u''u'v'}{L(u)^5},
\end{equation}
therefore
\begin{equation}\label{Dkap0}
D\kappa_1\left(\frac{n-1}{\eta}\right)[v]=-\left(\frac{\eta}{n-1}\right)^2v,\ \ \ D\kappa_n\left(\frac{n-1}{\eta}\right)[v]=-v''.
\end{equation}
Substituting this into (\ref{DF}) we find:
\begin{equation}\label{DF0}
\left.D\left(F\left(\kapsurf{u}\right)\right)\right|_{u=\frac{n-1}{\eta}}[v]=-\left(F_n(\eta)v''+\frac{\eta^2F_1(\eta)}{n-1}v\right).
\end{equation}
By (\ref{D1Gbar0}) this gives:
\begin{equation}\label{D1Gbar02}
D_1\bar{G}_{\eta_0}(0,\eta)[\bar{v}]=F_n(\eta)\bar{v}''+\frac{\eta^2F_1(\eta)}{(n-1)}\bar{v}.
\end{equation}
This map is a bijection from $h_{e,0}^{2,\alpha}\left(\Sdp\right)$ to $h_{e,0}^{0,\alpha}\left(\Sdp\right)$ except if $\eta=\frac{m\pi}{d}\sqrt{\frac{(n-1)F_n(\eta)}{F_1(\eta)}}$ for some $m\in\mathbb{N}$ or $F_1(\eta)=F_n(\eta)=0$. Thus bifurcation can only occur at these points, such as when $\eta=\frac{n-1}{R_{crit}}$.

When $\eta=\frac{m\pi}{d}\sqrt{\frac{(n-1)F_n(\eta)}{F_1(\eta)}}$ it is easily seen that
\begin{align*}
N\left(D_1\bar{G}_{\eta_0}(0,\eta)\right)=&\text{span}\left(\cos\left(\frac{m\pi z}{d}\right)\right)\\
R\left(D_1\bar{G}_{\eta_0}(0,\eta)\right)=&h_{e,0}^{0,\alpha}\left(\Sdp\right)/\cos\left(\frac{m\pi z}{d}\right).
\end{align*}
Also, from (\ref{D1Gbar02}),
\begin{align}\label{D12Gbar0}
D_{12}^2\bar{G}_{\eta_0}\left(0,\eta\right)[\bar{v}]=&F_{n}'(\eta)\bar{v}''+\left(\frac{2\eta F_1(\eta)}{n-1}+\frac{\eta^2F_1'(\eta)}{n-1}\right)\bar{v}\nonumber\\
=&\frac{F_{n}'(\eta)}{F_n(\eta)}D_1\bar{G}_{\eta_0}(0,\eta)[\bar{v}]+\frac{\eta}{n-1}\left(2F_1(\eta)+\eta\left(F_1'(\eta)-\frac{F_1(\eta)F_{n}'(\eta)}{F_n(\eta)}\right)\right)\bar{v}.
\end{align}
Therefore when $\eta$ satisfies $\eta=\frac{m\pi}{d}\sqrt{\frac{(n-1)F_n(\eta)}{F_1(\eta)}}$ we have
\begin{equation*}
D_{12}^2\bar{G}_{\eta_0}\left(0,\eta\right)\left[\cos\left(\frac{m\pi z}{d}\right)\right]= \frac{\eta}{n-1}\left(2F_1(\eta)+\eta\left(F_1'(\eta)-\frac{F_1(\eta)F_{n}'(\eta)}{F_n(\eta)}\right)\right)\cos\left(\frac{m\pi z}{d}\right).
\end{equation*}
To apply Theorem I.5.1 in \cite{Kielhofer12} we require that $D_{12}^2\bar{G}_{\eta_0}\left(0,\eta\right)\left[\cos\left(\frac{m\pi z}{d}\right)\right]\notin R\left(D_1\bar{G}_{\eta_0}(0,\eta)\right)$, which is equivalent to
\begin{equation}\label{BifCond}
2F_1(\eta)+\eta\left(F_1'(\eta)-\frac{F_1(\eta)F_{n}'(\eta)}{F_n(\eta)}\right)\neq0.
\end{equation}

To prove this is the case when $\eta=\frac{n-1}{R_{crit}}$, we notice that by assumption \descref{(A5)} the function $f(\eta)=\frac{\eta^2}{n-1}F_1(\eta)-\frac{\pi^2}{d^2}F_n(\eta)$ changes sign at $\eta=\frac{n-1}{R_{crit}}$, and hence $f'\left(\frac{n-1}{R_{crit}}\right)\neq0$. Calculating this derivative gives:
\begin{align*}
f'(\eta)=&\frac{2\eta}{n-1}F_1(\eta)+\frac{\eta^2F_1'(\eta)}{n-1}-\frac{\pi^2}{d^2}F_{n}'(\eta)\\
=&\frac{\eta}{n-1}\left(2F_1(\eta)+\eta\left(F_1'(\eta)-\frac{(n-1)^2\eta^2 F_1\left(\frac{n-1}{R_{crit}}\right)}{R_{crit}^2F_n\left(\frac{n-1}{R_{crit}}\right)}F_{n}'(\eta)\right)\right),
\end{align*}
so that (\ref{BifCond}) is satisfied when $\eta=\frac{n-1}{R_{crit}}$.
\end{proof}

\begin{corollary}\label{StatSol}
There exists a continuously differentiable family of nontrivial, axially symmetric hypersurfaces that are stationary solutions to the flow (\ref{WVPCF}), with assumptions \descref{(A1)}-\descref{(A5)} for $\tilde{R}=R_{crit}$, that includes the cylinder of radius $R_{crit}$, they are given by the profile curves $\tilde{\rho}(s):=\psi_{\eta_0}\left(\tilde{u}(s),\eta(s)\right)|_{[0,d]}$.
\end{corollary}

We now give a stronger corollary for when $F$ is a homogeneous function. This proves the first part of Corollary \ref{HomogCor}; the second part of which follows straight from Theorem \ref{MainThm}.

\begin{corollary}\label{StatSolHomog}
There exists a continuously differentiable family of nontrivial, axially symmetric hypersurfaces that are stationary solutions to the flow (\ref{WVPCF}), with assumptions \descref{(A1)}-\descref{(A5)*} satisfied at $\tilde{R}=\frac{n-1}{\eta_m}$ with $\eta_m:=\frac{m\pi}{d}\sqrt{\frac{(n-1)F_n}{F_1}}$, that includes the cylinder of radius $\frac{n-1}{\eta_m}$; they can each be represented by a profile curve: $\tilde{\rho}_m(s):=\psi_{\eta_m}\left(\tilde{u}_m(s),\eta_m(s)\right)|_{[0,d]}$.
\end{corollary}

\begin{proof}
Since $F$ is homogeneous, \descref{(A2)} reduces to $F_a\neq0$ for $a=1,n$. Therefore (\ref{D1Gbar02}) becomes
\begin{equation}
D_1\bar{G}_{\eta_0}(0,\eta)[\bar{v}]=\eta^{k-1}F_n\left(\bar{v}''+\frac{\eta^2F_1}{(n-1)F_n}\bar{v}\right),
\end{equation}
which is a bijection if and only if $\eta\neq\eta_m$. Thus bifurcation can only occur at $(0,\eta_m)$. Also, by using the relation $F_a'(\eta)=(k-1)\eta^{k-2}F_a$ for $a=1,n$, condition (\ref{BifCond}) reduces to $F_1\neq0$ and hence each of these points is a bifurcation point with bifurcation curve $(\tilde{u}_m(s),\eta_m(s))$.
\end{proof}

We will now consider the spectrum of $D_1\bar{G}_{\eta_0}(0,\eta)$. It is clear from (\ref{D1Gbar02}) that if $\eta<\eta_0$ ($\eta>\eta_0$, refer \descref{(A5)}) then the spectrum of $D_1\bar{G}_{\eta_0}(0,\eta)$ is contained in the negative real axis, this leads to the stability of the null solution which is special case of Theorem \ref{MainRCyl}. However, when $\eta=\eta_0$ the first eigenvalue becomes zero; we now determine how this eigenvalue behaves as we move from linearising about $\left(0,\eta_0\right)$ to linearising about points on the bifurcation curve.

To analyse the curve we make the following definitions:
\begin{equation*}
\hat{v}:=A\cos\left(\eta_0\sqrt{\frac{F_1(\eta_0)}{(n-1)F_n(\eta_0)}}z\right),\ \ A:=\left\|\cos\left(\eta_0\sqrt{\frac{F_1(\eta_0)}{(n-1)F_n(\eta_0)}}z\right)\right\|^{-1}_{h^{2,\alpha}},
\end{equation*}
\begin{equation*}
\tilde{v}:=B\cos\left(\eta_0\sqrt{\frac{F_1(\eta_0)}{(n-1)F_n(\eta_0)}}z\right),\ \ B:=\left\|\cos\left(\eta_0\sqrt{\frac{F_1(\eta_0)}{(n-1)F_n(\eta_0)}}z\right)\right\|^{-1}_{h^{0,\alpha}},
\end{equation*}
and
\begin{equation}\label{defvtils}
\tilde{v}^*[\bar{v}]:=\frac{2}{B}\fint_{\Sdp}\bar{v}\cos\left(\eta_0\sqrt{\frac{F_1(\eta_0)}{(n-1)F_n(\eta_0)}}z\right)\,dz,
\end{equation}
so that $\tilde{v}^*[\tilde{v}]=1$ and, by the self adjointness of $D_1\bar{G}_{\eta_0}(0,\eta_0)$ with respect to the $L^2$ inner product, $\tilde{v}^*\left[D_1\bar{G}_{\eta_0}\left(0,\eta_0\right)[v]\right]=0$ for all $v\in h_e^{2,\alpha}\left(\Sdp\right)$. For ease of notation we now drop all subscripts referencing $\eta_0$.

\begin{theorem}\label{BifShape}
Let \descref{(A1)}-\descref{(A5)} hold with $\tilde{R}=R_{crit}$. Then the bifurcation curve from Theorem \ref{BifPoints} has the following properties:
\begin{equation}\label{Deta}
\eta'(0)=0
\end{equation}
and
\begin{equation}\label{D2eta}
\eta''(0)=\frac{-\eta_0^3A^2}{12}\left(\frac{\mathscr{F}}{2F_1(\eta_0)+\eta_0F_{1}'(\eta_0)-\frac{\eta_0F_1(\eta_0)F_{n}'(\eta_0)}{F_n(\eta_0)}} -\frac{6\sum_{a=1}^n\frac{c_a\eta_0^{a}}{(n-1)^{a}}\left(\binom{n-2}{a-1}-\frac{F_1(\eta_0)}{F_n(\eta_0)}\binom{n-1}{a-1}\right)}{(n-1)\sum_{a=0}^n\frac{c_a\eta_0^a}{(n-1)^a}\binom{n-1}{a}}\right),
\end{equation}
where
\begin{align}
\mathscr{F}=&\frac{3\eta_0^2F_{1}''(\eta_0)}{(n-1)^2}-\frac{9\eta_0^2F_1(\eta_0)F_{n}''(\eta_0)}{(n-1)^2F_n(\eta_0)}+\frac{9\eta_0^2F_1(\eta_0)^2F_{nn}'(\eta_0)}{(n-1)^2F_n(\eta_0)^2}-\frac{3\eta_0^2F_1(\eta_0)^3F_{nnn}(\eta_0)}{(n-1)^2F_n(\eta_0)^3} \nonumber\\
&+\frac{\eta_0^2F_{1}'(\eta_0)^2}{(n-1)^2F_1(\eta_0)}-\frac{7\eta_0^2F_{1}'(\eta_0)F_{n}'(\eta_0)}{(n-1)^2F_n(\eta_0)}+\frac{5\eta_0^2F_1(\eta_0)F_{1}'(\eta_0)F_{nn}(\eta_0)}{(n-1)^2F_n(\eta_0)^2}\nonumber\\
&+\frac{10\eta_0^2F_1(\eta_0)F_{n}'(\eta_0)^2}{(n-1)^2F_n(\eta_0)^2}-\frac{13\eta_0^2F_1(\eta_0)^2F_{n}'(\eta_0)F_{nn}(\eta_0)}{(n-1)^2F_n(\eta_0)^3}+\frac{4\eta_0^2F_1(\eta_0)^3F_{nn}(\eta_0)^2}{(n-1)^2F_n(\eta_0)^4} \nonumber\\
&+\frac{2(3n+8)\eta_0F_{1}'(\eta_0)}{(n-1)^2}-\frac{4\eta_0F_1(\eta_0)F_{1}'(\eta_0)}{(n-1)F_n(\eta_0)}-\frac{2(3n+13)\eta_0F_1(\eta_0)F_{n}'(\eta_0)}{(n-1)^2F_n(\eta_0)}\nonumber\\
&+\frac{2\eta_0F_1(\eta_0)^2F_{n}'(\eta_0)}{(n-1)F_n(\eta_0)^2}+\frac{10\eta_0F_1(\eta_0)^2F_{nn}(\eta_0)}{(n-1)^2F_n(\eta_0)^2}+\frac{2\eta_0F_1(\eta_0)^3F_{nn}(\eta_0)}{(n-1)F_n(\eta_0)^3} +\frac{2(6n+5)F_1(\eta_0)}{(n-1)^2}\nonumber\\
&+\frac{4F_1(\eta_0)^2}{(n-1)F_n(\eta_0)}-\frac{2F_1(\eta_0)^3}{F_n(\eta_0)^2},\nonumber
\end{align}
and $F_{a}(\eta)=\frac{\partial F}{\partial\kappa_a}\left(\kapsurf{\frac{n-1}{\eta}}\right)$, $F_{nn}(\eta)=\frac{\partial^2 F}{\partial\kappa_n^2}\left(\kapsurf{\frac{n-1}{\eta}}\right)$ and $F_{nnn}(\eta)=\frac{\partial^3 F}{\partial\kappa_n^3}\left(\kapsurf{\frac{n-1}{\eta}}\right)$.
\end{theorem}

\begin{remark}\label{FderivRem}
Note that the derivatives that appear in equation (\ref{D2eta}) can be expanded in terms of the speed function as follows:
\begin{equation*}
F_1'(\eta)=\frac{1}{n-1}\left(\frac{\partial^2F}{\partial\kappa_1^2}\left(\kapsurf{\frac{n-1}{\eta}}\right)+(n-2)\frac{\partial^2F}{\partial\kappa_1\partial\kappa_2}\left(\kapsurf{\frac{n-1}{\eta}}\right)\right),
\end{equation*}
\begin{equation*}
F_1''(\eta)=\frac{1}{(n-1)^2}\left(\frac{\partial^3F}{\partial\kappa_1^3}\left(\kapsurf{\frac{n-1}{\eta}}\right) +2(n-2)\frac{\partial^3F}{\partial\kappa_1^2\partial\kappa_2}\left(\kapsurf{\frac{n-1}{\eta}}\right) +(n-2)(n-3)\frac{\partial^3F}{\partial\kappa_1\partial\kappa_2\partial\kappa_3}\left(\kapsurf{\frac{n-1}{\eta}}\right)\right),
\end{equation*}
\begin{equation*}
F_n'(\eta)=\frac{\partial^2F}{\partial\kappa_1\partial\kappa_n}\left(\kapsurf{\frac{n-1}{\eta}}\right),\ \ F_n''(\eta)=\frac{1}{n-1}\left(\frac{\partial^3F}{\partial\kappa_1^2\partial\kappa_n}\left(\kapsurf{\frac{n-1}{\eta}}\right) +(n-2)\frac{\partial^3F}{\partial\kappa_1\partial\kappa_2\partial\kappa_n}\left(\kapsurf{\frac{n-1}{\eta}}\right)\right),
\end{equation*}
\begin{equation*}
 F_{nn}'(\eta)=\frac{\partial^3F}{\partial\kappa_n^2\partial\kappa_1}\left(\kapsurf{\frac{n-1}{\eta}}\right).
\end{equation*}
\end{remark}

\begin{proof}
These formulas come from standard calculations using equations (I.6.3), (I.6.8) and (I.6.11) from \cite{Kielhofer12}:
\begin{equation}\label{Deta1}
\eta'(0)=\frac{-\tilde{v}^*\left[D_{11}^2\bar{G}(0,\eta_0)\left[\hat{v},\hat{v}\right]\right]}{2\tilde{v}^*\left[D_{12}^2\bar{G}(0,\eta_0)\left[\hat{v}\right]\right]},
\end{equation}
\begin{equation}\label{D2eta2}
\eta''(0)=\frac{-\tilde{v}^*\left[D_{111}^3\bar{G}(0,\eta_0)\left[\hat{v},\hat{v},\hat{v}\right]+3D_{11}^2\bar{G}(0,\eta_0)\left[\hat{v},\tilde{w}\right]\right]}{3\tilde{v}^*\left[D_{12}^2\bar{G}(0,\eta_0)\left[\hat{v}\right]\right]},
\end{equation}
where $\tilde{w}$ is the solution to
\begin{equation}\label{defw}
D_1\bar{G}(0,\eta_0)[\bar{w}]=-D_{11}^2\bar{G}(0,\eta_0)\left[\hat{v},\hat{v}\right],
\end{equation}
such that $\tilde{v}^*[\tilde{w}]=0$. Note that (\ref{D2eta2}) is only the correct formula when $\eta'(0)=0$.

By linearising (\ref{D1Gbar}) we have
\begin{align}\label{D11Gbar}
D_{11}^2\bar{G}(\bar{u},\eta)[\bar{v},\bar{w}]=&P_0\left[D^2G\left(\psi(\bar{u},\eta)\right)\left[D_1\psi(\bar{u},\eta)[\bar{v}],D_1\psi(\bar{u},\eta)[\bar{w}]\right]\right]\\
&+P_0\left[DG\left(\psi(\bar{u},\eta)\right)\left[D_{11}^2\psi(\bar{u},\eta)[\bar{v},\bar{w}]\right]\right],\nonumber
\end{align}
and from (\ref{Dpsi}) we find that
\begin{align}\label{DDpsi0}
D_{11}^2\psi(0,\eta_0)[\bar{v},\bar{w}]=&\Xi\left(\kapsurf{\frac{n-1}{\eta_0}}\right)^{-1}\fint_{\Sdp}\bar{v}\left(\frac{\eta_0^2}{n-1}\frac{\partial\Xi}{\partial\kappa_1}\left(\kapsurf{\frac{n-1}{\eta_0}}\right)\bar{w} +\frac{\partial\Xi}{\partial\kappa_n}\left(\kapsurf{\frac{n-1}{\eta_0}}\right)\bar{w}''\right)\,dz\\
& -\eta_0\fint_{\Sdp}\bar{v}\bar{w}\,dz.\nonumber
\end{align}
However at this stage it is only important that this is a constant function, so we associate it to its corresponding real number; in fact it is clear from (\ref{Dpsi}) that $D_{11}^2\psi(\bar{u},\eta)[\bar{v},\bar{w}]$ will be a constant function for any $(\bar{u},\eta)$. Using this, (\ref{D11Gbar}) simplifies to
\begin{equation*}
D_{11}^2\bar{G}(0,\eta_0)=P_0\left[D^2G\left(\frac{n-1}{\eta_0}\right)[\bar{v},\bar{w}]+D_{11}^2\psi(0,\eta_0)[\bar{v},\bar{w}]DG\left(\frac{n-1}{\eta_0}\right)[1]\right].
\end{equation*}
From (\ref{DF0}) we find $\left.D\left(F\left(\kapsurf{u}\right)\right)\right|_{u=\frac{n-1}{\eta_0}}[1]=-\frac{\eta_0^{2}}{n-1}F_1(\eta)$ and hence, by (\ref{DG0}), $DG\left(\frac{n-1}{\eta_0}\right)[1]=0$. Therefore
\begin{equation*}
D_{11}^2\bar{G}(0,\eta_0)=P_0\left[D^2G\left(\frac{n-1}{\eta_0}\right)[\bar{v},\bar{w}]\right].
\end{equation*}

Using (\ref{DG}) we have
\begin{align}\label{DDG}
D^2G(u)[v,w]=&\frac{w'v'G(u)+u'\left(v'DG(u)[w]+w'DG(u)[v]\right)}{L(u)^2} -\frac{3u'^2v'w'G(u)}{L(u)^4}\\
& +L(u)\left(\int_{\Sdp}D^2W(u)[v,w]\,dz-D^2\left(F\left(\kapsurf{u}\right)\right)[v,w]\right),\nonumber
\end{align}

which reduces to
\begin{equation*}
D^2G\left(\frac{n-1}{\eta_0}\right)[v,w]=\int_{\Sdp}D^2W\left(\frac{n-1}{\eta_0}\right)[v,w]\,dz-\left.D^2\left(F\left(\kapsurf{u}\right)\right)\right|_{u=\frac{n-1}{\eta_0}}[v,w].
\end{equation*}
So taking the projection gives:
\begin{equation*}
D_{11}^2\bar{G}(0,\eta_0)[\bar{v},\bar{w}]=\fint_{\Sdp}\left.D^2\left(F\left(\kapsurf{u}\right)\right)\right|_{u=\frac{n-1}{\eta_0}}[\bar{v},\bar{w}]\,dz -\left.D^2\left(F\left(\kapsurf{u}\right)\right)\right|_{u=\frac{n-1}{\eta_0}}[\bar{v},\bar{w}].
\end{equation*}

The second linearisation of $F\left(\kapsurf{u}\right)$ is given by:
\begin{align}\label{DDF}
D^2\left(F\left(\kapsurf{u}\right)\right)[v,w]=&(n-1)\left(\frac{\partial^2F}{\partial\kappa_1^2}\left(\kapsurf{u}\right) +(n-2)\frac{\partial^2F}{\partial\kappa_1\partial\kappa_2}\left(\kapsurf{u}\right)\right)D\kappa_1(u)[v]D\kappa_1(u)[w]\\
&+(n-1)\frac{\partial^2F}{\partial\kappa_1\partial\kappa_n}\left(\kapsurf{u}\right)\left(D\kappa_1(u)[v]D\kappa_n(u)[w]+D\kappa_n(u)[v]D\kappa_1(u)[w]\right)\nonumber\\
& +\frac{\partial^2F}{\partial\kappa_n^2}\left(\kapsurf{u}\right)D\kappa_n(u)[v]D\kappa_n(u)[w] +(n-1)\frac{\partial F}{\partial\kappa_1}\left(\kapsurf{u}\right)D^2\kappa_1(u)[v,w]\nonumber\\
&   +\frac{\partial F}{\partial\kappa_n}\left(\kapsurf{u}\right)D^2\kappa_n(u)[v,w].\nonumber
\end{align}
From (\ref{Dkap}) we obtain
\begin{equation}\label{DDkap}
\begin{array}{c}\normalsize{D^2\kappa_1(u)[v,w]=\frac{2vw}{u^3L(u)}+\frac{u'(vw'+v'w)}{u^2L(u)^3}-\frac{v'w'}{uL(u)^3}+\frac{3u'^2v'w'}{uL(u)^5}},\\ D^2\kappa_n(u)[v,w]=\frac{3(u'v''w'+u'v'w''+u''v'w')}{L(u)^5}-\frac{15u''u'^2v'w'}{L(u)^7},\end{array}
\end{equation}
therefore
\begin{equation}\label{DDkap0}
D^2\kappa_1\left(\frac{n-1}{\eta_0}\right)[v,w]=\frac{2\eta_0^3vw}{(n-1)^3} -\frac{\eta_0v'w'}{n-1},\ \ D^2\kappa_n\left(\frac{n-1}{\eta_0}\right)[v,w]=0,
\end{equation}
and combining this with (\ref{Dkap0}), (\ref{DDF}) and Remark \ref{FderivRem} we have
\begin{align}\label{DDF0}
\left.D^2\left(F\left(\kapsurf{u}\right)\right)\right|_{u=\frac{n-1}{\eta_0}}[v,w]=&\frac{\eta_0^{3}\left(\eta_0F_{1}'(\eta_0)+2F_1(\eta_0)\right)}{(n-1)^{2}}vw +\frac{\eta_0^{2}F_{n}'(\eta_0)}{n-1}\left(vw''+v''w\right)\\
& +F_{nn}(\eta_0)v''w'' -\eta_0F_1(\eta_0)v'w'.\nonumber
\end{align}
So by using the formula for $\hat{v}$ we obtain
\begin{align}\label{DDF0vv}
\left.D^2\left(F\left(\kapsurf{u}\right)\right)\right|_{u=\frac{n-1}{\eta_0}}[\hat{v},\hat{v}]=&\frac{\eta_0^{3}A^2}{n-1}\left(\mathscr{F}_1\cos^2\left(\eta_0\sqrt{\frac{F_1(\eta_0)}{(n-1)F_n(\eta_0)}}z\right)\right.\nonumber\\
&\hspace{1.0cm}\left. -\frac{F_1(\eta_0)^2}{F_n(\eta_0)}\sin^2\left(\eta_0\sqrt{\frac{F_1(\eta_0)}{(n-1)F_n(\eta_0)}}z\right)\right)\nonumber\\
=&\frac{\eta_0^{3}A^2}{2(n-1)}\left(\mathscr{F}_1-\frac{F_1(\eta_0)^2}{F_n(\eta_0)}\right.\nonumber\\
&\hspace{1.5cm}\left.+\left(\mathscr{F}_1+\frac{F_1(\eta_0)^2}{F_n(\eta_0)}\right)\cos\left(2\eta_0\sqrt{\frac{F_1(\eta_0)}{(n-1)F_n(\eta_0)}}z\right)\right),\nonumber
\end{align}
where
\begin{equation*}
\mathscr{F}_1:=\frac{\eta_0 F_{1}'(\eta_0)}{n-1} -\frac{2\eta_0F_1(\eta_0)F_{n}'(\eta_0)}{(n-1)F_n(\eta_0)} +\frac{\eta_0F_1(\eta_0)^2F_{nn}(\eta_0)}{(n-1)F_n(\eta_0)^2}+\frac{2F_1(\eta_0)}{n-1}.
\end{equation*}
Hence
\begin{equation}\label{D11Gbar0vv}
D_{11}^2\bar{G}(0,\eta_0)[\hat{v},\hat{v}]=-\frac{\eta_0^{3}A^2}{2(n-1)}\left(\mathscr{F}_1+\frac{F_1(\eta_0)^2}{F_n(\eta_0)}\right)\cos\left(2\eta_0\sqrt{\frac{F_1(\eta_0)}{(n-1)F_n(\eta_0)}}z\right).
\end{equation}
From the formula for $\tilde{v}^*$ in (\ref{defvtils}) we easily see that $\tilde{v}^*\left[D_{11}^2\bar{G}(0,\eta_0)[\hat{v},\hat{v}]\right]=0$ and hence from (\ref{Deta1}) we get the first result.

We now turn our attention to the second derivative of $\eta(s)$. We see from (\ref{D1Gbar02}) and (\ref{D11Gbar0vv}) that the solution to (\ref{defw}) is of the form $\tilde{\omega}=C\cos\left(2\eta_0\sqrt{\frac{F_1(\eta_0)}{(n-1)F_n(\eta_0)}}z\right)$ and in fact
\begin{equation*}
C=-\frac{\eta_0A^2}{6F_1(\eta_0)}\left(\mathscr{F}_1+\frac{F_1(\eta_0)^2}{F_n(\eta_0)}\right).
\end{equation*}
Now if we define:
\begin{equation*}
\mathscr{F}_2=\frac{\eta_0F_{1}'(\eta_0)}{n-1} -\frac{5\eta_0F_1(\eta_0)F_{n}'(\eta_0)}{(n-1)F_n(\eta_0)} +\frac{4\eta_0F_1(\eta_0)^2F_{nn}(\eta_0)}{(n-1)F_n(\eta_0)^2}+\frac{2F_1(\eta_0)}{n-1},
\end{equation*}
we can use (\ref{DDF0}) to obtain
\begin{align}
\left.D^2\left(F\left(\kapsurf{u}\right)\right)\right|_{u=\frac{n-1}{\eta_0}}[\hat{v},\tilde{w}]=& \frac{\eta_0^{3}AC}{n-1}\left(\mathscr{F}_2\cos\left(\eta_0\sqrt{\frac{F_1(\eta_0)}{(n-1)F_n(\eta_0)}}z\right)\cos\left(2\eta_0\sqrt{\frac{F_1(\eta_0)}{(n-1)F_n(\eta_0)}}z\right)\right.\nonumber\\
&\hspace{1.0cm}\left. -\frac{2F_1(\eta_0)^2}{F_n(\eta_0)}\sin\left(\eta_0\sqrt{\frac{F_1(\eta_0)}{(n-1)F_n(\eta_0)}}z\right)\sin\left(2\eta_0\sqrt{\frac{F_1(\eta_0)}{(n-1)F_n(\eta_0)}}z\right)\right)\nonumber\\
=&\frac{\eta_0^{3}AC}{2(n-1)}\left(\left(\mathscr{F}_2-\frac{2F_1(\eta_0)^2}{F_n(\eta_0)}\right)\cos\left(\eta_0\sqrt{\frac{F_1(\eta_0)}{(n-1)F_n(\eta_0)}}z\right)\right.\nonumber\\
&\hspace{1.3cm}\left. +\left(\mathscr{F}_2+\frac{2F_1(\eta_0)^2}{F_n(\eta_0)}\right)\cos\left(3\eta_0\sqrt{\frac{F_1(\eta_0)}{(n-1)F_n(\eta_0)}}z\right)\right).\nonumber
\end{align}
Thus $\fint_{\Sdp}\left.D^2\left(F\left(\kapsurf{u}\right)\right)\right|_{u=\frac{n-1}{\eta_0}}[\hat{v},\tilde{w}]\,dz=0$, $D_{11}^2\bar{G}(0,\eta_0)[\hat{v},\tilde{w}]=-\left.D^2\left(F\left(\kapsurf{u}\right)\right)\right|_{u=\frac{n-1}{\eta_0}}[\hat{v},\tilde{w}]$ and
\begin{equation}\label{vD11Gbar0vw}
\tilde{v}^*\left[D_{11}^2\bar{G}(0,\eta_0)[\hat{v},\tilde{w}]\right]=\frac{\eta_0^{4}A^3}{12(n-1)F_1(\eta_0)B}\left(\mathscr{F}_1+\frac{F_1(\eta_0)^2}{F_n(\eta_0)}\right)\left(\mathscr{F}_2-\frac{2F_1(\eta_0)^2}{F_n(\eta_0)}\right).
\end{equation}

From (\ref{D11Gbar}) it is easily seen that
\begin{align*}
D_{111}^3\bar{G}(0,\eta_0)[\hat{v},\hat{v},\hat{v}]=&P_0\left[D^3G\left(\frac{n-1}{\eta_0}\right)[\hat{v},\hat{v},\hat{v}]+3D^2G\left(\frac{n-1}{\eta_0}\right)\left[\hat{v},D_{11}^2\psi(0,\eta_0)[\hat{v},\hat{v}]\right]\right.\\
&\hspace{0.4cm}\left.+DG\left(\frac{n-1}{\eta_0}\right)\left[D_{111}^3\psi(0,\eta_0)[\hat{v},\hat{v},\hat{v}]\right]\right]\\
=&P_0\left[D^3G\left(\frac{n-1}{\eta_0}\right)[\hat{v},\hat{v},\hat{v}]+3D_{11}^2\psi(0,\eta_0)[\hat{v},\hat{v}]D^2G\left(\frac{n-1}{\eta_0}\right)\left[\hat{v},1\right]\right],
\end{align*}
where we have used that both $D_{11}^2\psi(0,\eta_0)[\hat{v},\hat{v}]$ and $D_{111}^3\psi(0,\eta_0)[\hat{v},\hat{v},\hat{v}]$ are constant functions, see the comment after (\ref{DDpsi0}), and also that $DG\left(\frac{n-1}{\eta_0}\right)\left[1\right]=0$. From (\ref{DDF0}) we see
\begin{align}
\left.D^2\left(F\left(\kapsurf{u}\right)\right)\right|_{u=\frac{n-1}{\eta_0}}[\hat{v},1]=&\frac{\eta_0^{3}A\left(\eta_0F_{1}'(\eta_0)-\frac{\eta_0F_1(\eta_0)F_{n}'(\eta_0)}{F_n(\eta_0)}+2F_1(\eta_0)\right)}{(n-1)^2}\cos\left(\eta_0\sqrt{\frac{F_1(\eta_0)}{(n-1)F_n(\eta_0)}}z\right),\nonumber
\end{align}
hence
\begin{align}
P_0\left[D^2G\left(\frac{n-1}{\eta_0}\right)[\hat{v},1]\right]=&\fint_{\Sdp}\left.D^2\left(F\left(\kapsurf{u}\right)\right)\right|_{u=\frac{n-1}{\eta_0}}[\hat{v},1]\,dz -\left.D^2\left(F\left(\kapsurf{u}\right)\right)\right|_{u=\frac{n-1}{\eta_0}}[\hat{v},1]\nonumber\\
=&-\frac{\eta_0^{3}A\left(\eta_0F_{1}'(\eta_0)-\frac{\eta_0F_1(\eta_0)F_{n}'(\eta_0)}{F_n(\eta_0)}+2F_1(\eta_0)\right)}{(n-1)^2}\cos\left(\eta_0\sqrt{\frac{F_1(\eta_0)}{(n-1)F_n(\eta_0)}}z\right).\nonumber
\end{align}
From (\ref{DDpsi0}) we also see that
\begin{align}
D_{11}^2\psi(0,\eta_0)[\hat{v},\hat{v}]=&\frac{\eta_0^2A^2}{(n-1)\Xi\left(\kapsurf{\frac{n-1}{\eta_0}}\right)}\fint_{\Sdp}\left(\frac{\partial\Xi}{\partial\kappa_1}\left(\kapsurf{\frac{n-1}{\eta_0}}\right)-\frac{F_1(\eta_0)}{F_n(\eta_0)}\frac{\partial\Xi}{\partial\kappa_n}\left(\kapsurf{\frac{n-1}{\eta_0}}\right)\right)\cos^2\left(\eta_0\sqrt{\frac{F_1(\eta_0)}{(n-1)F_n(\eta_0)}}z\right)\,dz\nonumber\\
&-\eta_0A^2\fint_{\Sdp}\cos^2\left(\eta_0\sqrt{\frac{F_1(\eta_0)}{(n-1)F_n(\eta_0)}}z\right)\,dz\nonumber\\
=&\frac{\eta_0^2A^2}{2(n-1)\Xi\left(\kapsurf{\frac{n-1}{\eta_0}}\right)}\left(\frac{\partial\Xi}{\partial\kappa_1}\left(\kapsurf{\frac{n-1}{\eta_0}}\right)-\frac{F_1(\eta_0)}{F_n(\eta_0)}\frac{\partial\Xi}{\partial\kappa_n}\left(\kapsurf{\frac{n-1}{\eta_0}}\right)\right) -\frac{\eta_0A^2}{2}.\nonumber
\end{align}

Lastly, the third linearisation of $G$ can be calculated from (\ref{DDG}):
\begin{align}\label{DDDG0vvv}
D^3G\left(\frac{n-1}{\eta_0}\right)[\hat{v},\hat{v},\hat{v}]=&\int_{\Sdp}D^3W\left(\frac{n-1}{\eta_0}\right)[\hat{v},\hat{v},\hat{v}]\,dz - \left.D^3\left(F\left(\kapsurf{u}\right)\right)\right|_{u=\frac{n-1}{\eta_0}}[\hat{v},\hat{v},\hat{v}],
\end{align}
where we have used (\ref{D1Gbar0}) and the definition of $\hat{v}$ as being in the null space of $D_1\bar{G}(0,\eta_0)$.

We calculate $\left.D^3\left(F\left(\kapsurf{u}\right)\right)\right|_{u=\frac{n-1}{\eta_0}}[\hat{v},\hat{v},\hat{v}]$ using Remark \ref{FderivRem} along with (\ref{Dkap0}), (\ref{DDkap0}) and
\begin{equation*}
D^3\kappa_1\left(\frac{n-1}{\eta_0}\right)[\hat{v},\hat{v},\hat{v}]=\frac{-6\eta_0^4}{(n-1)^4}\hat{v}^3+\frac{3\eta_0^2}{(n-1)^2}\hat{v}\hat{v}'^2,\ \ D^3\kappa_n\left(\frac{n-1}{\eta_0}\right)[\hat{v},\hat{v},\hat{v}]=9\hat{v}''\hat{v}'^2,
\end{equation*}
to obtain
\begin{align}
\left.D^3\left(F\left(\kapsurf{u}\right)\right)\right|_{u=\frac{n-1}{\eta_0}}[\hat{v},\hat{v},\hat{v}] =&-\frac{\eta_0^{6}F_{1}''(\eta_0)}{(n-1)^{3}}\hat{v}^3 -\frac{3\eta_0^{4}F_{n}''(\eta_0)}{(n-1)^{2}}\hat{v}^2\hat{v}'' -\frac{3\eta_0^{2}F_{nn}'(\eta_0)}{n-1}\hat{v}\hat{v}''^2 -F_{nnn}(\eta_0)\hat{v}''^3\nonumber\\
& -\frac{3\eta_0^{3}F_{1}'(\eta_0)}{n-1}\left(\frac{2\eta_0^2}{(n-1)^2}\hat{v}^2-\hat{v}'^2\right)\hat{v} -3\eta_0F_{n}'(\eta_0)\left(\frac{2\eta_0^2}{(n-1)^2}\hat{v}^2-\hat{v}'^2\right)\hat{v}''\nonumber\\
& +\frac{3\eta_0^{2}F_1(\eta_0)}{n-1}\left(\hat{v}\hat{v}'^2-\frac{2\eta_0^2}{(n-1)^2}\hat{v}^3\right) +9F_n(\eta_0)\hat{v}''\hat{v}'^2\nonumber\\
=&-\frac{\eta_0^{4}A^3}{(n-1)^2}\left(\mathscr{F}_3\cos^3\left(\eta_0\sqrt{\frac{F_1(\eta_0)}{(n-1)F_n(\eta_0)}}z\right)\right.\nonumber\\
&\hspace{1.7cm} \left.+3\mathscr{F}_4\sin^2\left(\eta_0\sqrt{\frac{F_1(\eta_0)}{(n-1)F_n(\eta_0)}}z\right)\cos\left(\eta_0\sqrt{\frac{F_1(\eta_0)}{(n-1)F_n(\eta_0)}}z\right)\right)\nonumber\\
=&-\frac{\eta_0^{4}A^3}{4(n-1)^{2}}\left(3\left(\mathscr{F}_3+\mathscr{F}_4\right)\cos\left(\eta_0\sqrt{\frac{F_1(\eta_0)}{(n-1)F_n(\eta_0)}}z\right)\right.\nonumber\\
&\hspace{1.9cm} \left.+\left(\mathscr{F}_3-3\mathscr{F}_4\right)\cos\left(3\eta_0\sqrt{\frac{F_1(\eta_0)}{(n-1)F_n(\eta_0)}}z\right)\right),\nonumber
\end{align}
where
\begin{align*}
\mathscr{F}_3:=&\frac{\eta_0^2F_{1}''(\eta_0)}{n-1} -\frac{3\eta_0^2F_1(\eta_0)F_{n}''(\eta_0)}{(n-1)F_n(\eta_0)} +\frac{3\eta_0^2F_1(\eta_0)^2F_{nn}'(\eta_0)}{(n-1)F_n(\eta_0)^2} -\frac{\eta_0^2F_1(\eta_0)^3F_{nnn}(\eta_0)}{(n-1)F_n(\eta_0)^3}\\
& +\frac{6\eta_0F_{1}'(\eta_0)}{n-1} -\frac{6\eta_0F_1(\eta_0)F_{n}'(\eta_0)}{(n-1)F_n(\eta_0)} +\frac{6F_1(\eta_0)}{n-1}
\end{align*}
and
\begin{equation*}
\mathscr{F}_4:=\frac{-\eta_0F_1(\eta_0)F_{1}'(\eta_0)}{F_n(\eta_0)} +\frac{\eta_0F_1(\eta_0)^2F_{n}'(\eta_0)}{F_n(\eta_0)^2} +\frac{2F_1(\eta_0)^2}{F_n(\eta_0)}.
\end{equation*}

Combining this with (\ref{DDDG0vvv}) we conclude
\begin{align}
P_0\left[D^3G\left(\frac{n-1}{\eta_0}\right)[\hat{v},\hat{v},\hat{v}]\right]=&\fint_{\Sdp}\left.D^3\left(F\left(\kapsurf{u}\right)\right)\right|_{u=\frac{n-1}{\eta_0}}[\hat{v},\hat{v},\hat{v}] \,dz -\left.D^3\left(F\left(\kapsurf{u}\right)\right)\right|_{u=\frac{n-1}{\eta_0}}[\hat{v},\hat{v},\hat{v}] \nonumber\\
=&\frac{\eta_0^{4}A^3}{4(n-1)^{2}}\left(3\left(\mathscr{F}_3+\mathscr{F}_4\right)\cos\left(\eta_0\sqrt{\frac{F_1(\eta_0)}{(n-1)F_n(\eta_0)}}z\right)\right.\nonumber\\
&\hspace{1.7cm} \left.+\left(\mathscr{F}_3-3\mathscr{F}_4\right)\cos\left(3\eta_0\sqrt{\frac{F_1(\eta_0)}{(n-1)F_n(\eta_0)}}z\right)\right),\nonumber
\end{align}
and
\begin{align}
\tilde{v}^*\left[D_{111}^3\bar{G}(0,\eta_0)[\hat{v},\hat{v},\hat{v}]\right]=&\frac{3\eta_0^4A^3}{4(n-1)^2B}\left(\mathscr{F}_3+\mathscr{F}_4 +2\eta_0F_{1}'(\eta_0) -\frac{2\eta_0F_1(\eta_0)F_{1}'(\eta_0)}{F_n(\eta_0)} +4F_1(\eta_0)\right.\nonumber\\
&\hspace{1.7cm}\left.-\frac{2\eta_0\left(\frac{\partial\Xi}{\partial\kappa_1}\left(\kapsurf{\frac{n-1}{\eta_0}}\right)-\frac{F_1(\eta_0)}{F_n(\eta_0)}\frac{\partial\Xi}{\partial\kappa_n}\left(\kapsurf{\frac{n-1}{\eta_0}}\right)\right)}{(n-1)\Xi\left(\kapsurf{\frac{n-1}{\eta_0}}\right)}\left(\eta_0F_{1}'(\eta_0)-\frac{\eta_0F_1(\eta_0)F_{n}'(\eta_0)}{F_n(\eta_0)}+2F_1(\eta_0)\right)\right).\nonumber
\end{align}
Combining this with (\ref{D12Gbar0}) and (\ref{vD11Gbar0vw}) into equation (\ref{D2eta2}) gives
\begin{align}
\eta''(0)=& \frac{-\eta_0^3A^2}{12F_1(\eta_0)\left(2F_1(\eta_0)+\eta_0F_{1}'(\eta_0)-\frac{\eta_0F_1(\eta_0)F_{n}'(\eta_0)}{F_n(\eta_0)}\right)}\left(\frac{3F_1(\eta_0)}{n-1}\left(\mathscr{F}_3+\mathscr{F}_4 +2\eta_0F_{1}'(\eta_0) -\frac{2\eta_0F_1(\eta_0)F_{1}'(\eta_0)}{F_n(\eta_0)} +4F_1(\eta_0)\right)\right.\\
&\hspace{4cm}\left.-\frac{6\eta_0F_1(\eta_0)\left(\frac{\partial\Xi}{\partial\kappa_1}\left(\kapsurf{\frac{n-1}{\eta_0}}\right)-\frac{F_1(\eta_0)}{F_n(\eta_0)}\frac{\partial\Xi}{\partial\kappa_n}\left(\kapsurf{\frac{n-1}{\eta_0}}\right)\right)}{(n-1)^2\Xi\left(\kapsurf{\frac{n-1}{\eta_0}}\right)}\left(\eta_0F_{1}'(\eta_0)-\frac{\eta_0F_1(\eta_0)F_{n}'(\eta_0)}{F_n(\eta_0)}+2F_1(\eta_0)\right) \right.\nonumber\\
&\hspace{4cm} \left.+\left(\mathscr{F}_1+\frac{F_1(\eta_0)^2}{F_n(\eta_0)}\right)\left(\mathscr{F}_2-\frac{2F_1(\eta_0)^2}{F_n(\eta_0)}\right)\right)\nonumber
\end{align}
The formula (\ref{D2eta}) then follows by expanding and using that due to axial symmetry, (\ref{EleDef}) can be rewritten as $E_b(\kapsurf{u})=\sum_{a=0}c_a\left(\binom{n-1}{a}\kappa_1^a+\binom{n-1}{a-1}\kappa_1^{a-1}\kappa_n\right)$, with the derivatives given by $\frac{\partial\Xi}{\partial\kappa_1}\left(\kapsurf{u}\right)=\sum_{a=1}^nc_a\left(\binom{n-2}{a-1}\kappa_1(u)^{a-1}+\binom{n-2}{a-2}\kappa_1(u)^{a-2}\kappa_n(u)\right)$ and $\frac{\partial\Xi}{\partial\kappa_n}\left(\kapsurf{u}\right)=\sum_{a=1}^nc_a\binom{n-1}{a-1}\kappa_1(u)^{a-1}$.
\end{proof}

\begin{remark}
In the case where $F$ is homogeneous we found a sequence of bifurcation points, $(0,\eta_m)$. In this case by setting $\eta_0=\eta_m$ the same analysis shows that equations (\ref{Deta}) and (\ref{D2eta}) are still the correct formula. However, the analysis is less relevant here since there is already a strictly positive eigenvalue for $D_1\bar{G}(0,\eta_m)$.
\end{remark}

We are now able to give a stability theorem for the full flow.
\begin{corollary}\label{GenStable}
Let \descref{(A1)}-\descref{(A5)} be satisfied with $\tilde{R}=R_{crit}$. If
\begin{equation}
\mathscr{F} -\frac{6\sum_{a=1}^n\frac{c_a\eta_0^{a}}{(n-1)^{a}}\left(\binom{n-2}{a-1}-\frac{F_1(\eta_0)}{F_n(\eta_0)}\binom{n-1}{a-1}\right)}{(n-1)\sum_{a=0}^n\frac{c_a\eta_0^a}{(n-1)^a}\binom{n-1}{a}}\left(2F_1(\eta_0)+\eta_0F_{1}'(\eta_0) -\frac{\eta_0F_1(\eta_0)F_{n}'(\eta_0)}{F_n(\eta_0)}\right)>0,
\end{equation}
then the stationary solutions to (\ref{WVPCF}) that are close to the cylinder of radius $R_{crit}$ are unstable equilibria.

Alternatively if
\begin{equation}\label{FCond}
\mathscr{F} -\frac{6\sum_{a=1}^n\frac{c_a\eta_0^{a}}{(n-1)^{a}}\left(\binom{n-2}{a-1}-\frac{F_1(\eta_0)}{F_n(\eta_0)}\binom{n-1}{a-1}\right)}{(n-1)\sum_{a=0}^n\frac{c_a\eta_0^a}{(n-1)^a}\binom{n-1}{a}}\left(2F_1(\eta_0)+\eta_0F_{1}'(\eta_0) -\frac{\eta_0F_1(\eta_0)F_{n}'(\eta_0)}{F_n(\eta_0)}\right)<0,
\end{equation}
then the stationary solutions to (\ref{WVPCF}) that are close to the cylinder of radius $R_{crit}$ are stable under axially symmetric, weighted-volume preserving perturbations. That is, there exists $\epsilon>0$ such that for any $s\in(0,\epsilon)$ there exists a neighbourhood, $Z_s\subset h_{\frac{d}{dz}}^{2,\alpha}\left([0,d]\right)$, of $\tilde{\rho}(s)$, such that for any $\rho_0\in Z_s$ with $WVol(\rho_0)=WVol(\tilde{\rho}(s))$, the flow (\ref{WVPCFGraph}), with orthogonal boundary condition, exists for all time and the solution $\rho(t)$ converges exponentially fast to $\tilde{\rho}(s)$ as $t\rightarrow\infty$.
\end{corollary}

\begin{proof}
We start by again noting that the eigenvalues of $D_1\bar{G}(0,\eta_0)$, except for the dominant one, lie in the open complex halfplane, $Re\left(\lambda\right)<0$. Through a perturbation argument this is also true for the operator $D_1\bar{G}(\tilde{u}(s),\eta(s))$ as long as $s$ is small. We now determine the sign of the dominant eigenvalue of $D_1\bar{G}(\tilde{u}(s),\eta(s))$ for $s$ small. By Proposition I.7.2 in \cite{Kielhofer12}, there exists $\epsilon\in(0,\delta)$ and a continuously differentiable curve:
\begin{equation*}
\{\lambda(s):|s|<\epsilon,\lambda_{0}=0\}\subset\Rone,
\end{equation*}
such that
\begin{equation*}
D_1\bar{G}(\tilde{u}(s),\eta(s))[\hat{v}+v(s)]=\lambda(s)(\hat{v}+v(s)),
\end{equation*}
where $v(s)$, for $|s|<\epsilon$, is a continuously differentiable curve in range of $D_1\bar{G}(\tilde{u}(s),\eta(s))$ satisfying $v(0)=0$. Also, since $\eta'(0)=0$, we can use equations (I.7.34), (I.7.40) and (I.7.45) in \cite{Kielhofer12} to conclude that $\frac{d\lambda}{ds}(0)=0$ and
\begin{align}\label{D2lambda}
\frac{d^2\lambda}{ds^2}(0)=&-2\tilde{v}^*\left[D_{12}^2\bar{G}(0,\eta_0)\left[\hat{v}\right]]\right]\eta''(0)\nonumber\\
=&\frac{\eta_0^4A^3}{6(n-1)B}\left(\mathscr{F} -\frac{6\sum_{a=1}^n\frac{c_a\eta_0^{a}}{(n-1)^{a}}\left(\binom{n-2}{a-1}-\frac{F_1(\eta_0)}{F_n(\eta_0)}\binom{n-1}{a-1}\right)}{(n-1)\sum_{a=0}^n\frac{c_a\eta_0^a}{(n-1)^a}\binom{n-1}{a}}\left(2F_1(\eta_0)+\eta_0F_{1}'(\eta_0) -\frac{\eta_0F_1(\eta_0)F_{n}'(\eta_0)}{F_n(\eta_0)}\right)\right).
\end{align}

In the first case we see from equation (\ref{D2lambda}) that $\lambda(0)=0$ is a local minimum of $\lambda(s)$ and hence, possibly making $\epsilon$ smaller, the eigenvalue $\lambda(s)$ is positive for $0<|s|<\epsilon$. We also note that $D_1\bar{G}(0,\eta_0)[\bar{v}]$ is the negative of an elliptic operator, so by Theorem 3.2.6 in \cite{HartleyPhD} it is a sectorial operator on the little-H\"older spaces. The perturbation result in Proposition 2.4.2 \cite{Lunardi95} then ensures that $D_1\bar{G}(\tilde{u}(s),\eta(s))$ is sectorial on the little-H\"older spaces for all $|s|<\epsilon$ (again possibly making $\epsilon$ smaller).

From (\ref{lHInterp}) we know that the little-H\"older spaces are interpolation spaces and we can apply Theorem 9.1.7 in \cite{Lunardi95} to obtain a nontrivial backward solution, $\bar{u}(t)$, of (\ref{RedWVPCF}) with $\eta=\eta(s)$ such that:
\begin{equation*}
\left\|\bar{u}(t)-\tilde{u}(s)\right\|_{h^{2,\alpha}}\leq Ce^{\omega t},\ t\leq0,
\end{equation*}
where $C,\omega>0$. By setting $\rho(t):=\psi\left(\bar{u}(t),\eta(s)\right)|_{[0,d]}$ we obtain a nontrivial backward solution to (\ref{WVPCFGraph}) such that
\begin{align*}
\left\|\rho(t)-\tilde{\rho}(s)\right\|_{h^{2,\alpha}}=&\left\|\psi\left(\bar{u}(t),\eta(s)\right)|_{[0,d]}-\psi\left(\tilde{u}(s),\eta(s)\right)|_{[0,d]}\right\|_{h^{2,\alpha}}\\
\leq&\left\|\psi\left(\bar{u}(t),\eta(s)\right)-\psi\left(\tilde{u}(s),\eta(s)\right)\right\|_{h^{2,\alpha}}\\
\leq& C\left\|\bar{u}(t)-\tilde{u}(s)\right\|_{h^{2,\alpha}}\\
\leq &Ce^{\omega t},\ t\leq0,
\end{align*}
where we have used that $\psi$ is Lipschitz. Thus we have instability of the stationary solution.

However in the second case we see that $\lambda(0)=0$ is a local maximum of $\lambda(s)$ and hence the eigenvalue $\lambda(s)$ is negative for $0<|s|<\epsilon$. We can therefore prove stability of the hypersurface defined by $\rho(s)$ by applying Theorem 9.1.7 in \cite{Lunardi95}. There exist $C,r,\omega>0$ such that if $\|\bar{u}_0-\tilde{u}(s)\|_{h^{2,\alpha}}<r$ then the solution, $\bar{u}(t)$, of (\ref{RedWVPCF}) with $\eta=\eta(s)$ and initial condition $\bar{u}_0$ is defined for all $t\geq0$ and satisfies
\begin{equation}\label{Fordecay}
\left\|\bar{u}(t)-\tilde{u}(s)\right\|_{h^{2,\alpha}}+\left\|\bar{u}'(t)\right\|_{h^{0,\alpha}}\leq Ce^{-\omega t}\left\|\bar{u}_0-\tilde{u}(s)\right\|_{h^{2,\alpha}},\ t\geq0.
\end{equation}

By now considering $\rho_0$ such that $\left\|\rho_0-\tilde{\rho}(s)\right\|_{h^{2,\alpha}}<\frac{r}{4}$ and $WVol(\rho_0)=WVol(\tilde{\rho}(s))$, then we have
\begin{align*}
\left\|P_0[u_{\rho_0}]-\tilde{u}(s)\right\|_{h^{2,\alpha}}=&\left\|P_0\left[u_{\rho_0}-\psi\left(\tilde{u}(s),\eta(s)\right)\right]\right\|_{h^{2,\alpha}}\\
\leq&2\left\|u_{\rho_0}-\psi\left(\tilde{u}(s),\eta(s)\right)\right\|_{h^{2,\alpha}}\\
\leq&4\left\|\rho_0-\psi\left(\tilde{u}(s),\eta(s)\right)|_{[0,d]}\right\|_{h^{2,\alpha}}\\
<&r.
\end{align*}
So, by the above calculations, there is a solution of (\ref{RedWVPCF}) with $\eta=\eta(s)$ and $\bar{u}(0)=P_0[u_{\rho_0}]$, $\bar{u}(t)$, that satisfies (\ref{Fordecay}). By setting $\rho(t)=\psi\left(\bar{u}(t),\eta(s)\right)|_{[0,d]}$ we obtain a solution to (\ref{WVPCFGraph}) with $\rho(0)=\psi\left(P_0[u_{\rho_0}],\eta(s)\right)|_{[0,d]}=u_{\rho_0}|_{[0,d]}=\rho_0$ such that
\begin{align*}
\left\|\rho(t)-\tilde{\rho}(s)\right\|_{h^{2,\alpha}}=&\left\|\psi\left(\bar{u}(t),\eta(s)\right)|_{[0,d]}-\psi\left(\tilde{u}(s),\eta(s)\right)|_{[0,d]}\right\|_{h^{2,\alpha}}\\
\leq&\left\|\psi\left(\bar{u}(t),\eta(s)\right)-\psi\left(\tilde{u}(s),\eta(s)\right)\right\|_{h^{2,\alpha}}\\
\leq& C\left\|\bar{u}(t)-\tilde{u}(s)\right\|_{h^{2,\alpha}}
\end{align*}
Therefore from (\ref{Fordecay}):
\begin{align*}
\left\|\rho(t)-\tilde{\rho}(s)\right\|_{h^{2,\alpha}}\leq& Ce^{-\omega t}\left\|P_0[u_{\rho_0}]-\tilde{u}(s)\right\|_{h^{2,\alpha}},\ t\geq0.
\end{align*}
Thus the hypersurface defined by $\tilde{\rho}(s)$ is a stable stationary solution of (\ref{WVPCF}) under axially symmetric, weighted-volume preserving perturbations.
\end{proof}


\section{Mixed-Volume Preserving Mean Curvature Flow}\label{SecMVPMCF}
In this section we consider the specific case of the mixed-volume preserving mean curvature flow (including the classical volume preserving mean curvature flow). In this case we have $F\left(\bm{\kappa}\right)=\sum_{a=1}^n\kappa_a$ and $\Xi\left(\bm{\kappa}\right)=E_b\left(\bm{\kappa}\right)$ (i.e. $c_a=1$ for $a=b$ and $c_a=0$ otherwise). Therefore $k=1$, $F_1=F_n=1$, and $F_{nn}=F_{nnn}=0$. Note that for assumption \descref{(A3)} to be satisfied we must exclude the case $b=n$ but in the other cases, $b=1,\ldots,n-1$, it is satisfied for any $\tilde{R}\in\Rone^+$. The stationary solutions to the flow are CMC hypersurfaces and the family of (mostly) non-cylindrical stationary solutions found in Corollary \ref{StatSol} represent the unduloids, with the cylindrical element of the family having radius $\frac{d\sqrt{n-1}}{\pi}$. In this case condition (\ref{HomogCond}) reduces to
\begin{equation}\label{FCondMV}
\frac{-\left(n^3-(b+10)n^2+2(5b-1)n-2b(3b-4)\right)}{(n-b)}<0.
\end{equation}
Further cancellation occurs in the volume preserving ($b=0$) and surface area-preserving ($b=1$) cases. In both situations the condition reduces to $n^2-10n-2>0$. For $b\geq2$ the cubic $n^3-(b+10)n^2+2(5b-1)n-2b(3b-4)$ has a single real root, which is also positive, and two complex roots. We now define this root: for $b=2,\ldots,n-1$ the real root of $n^3-(b+10)n^2+2(5b-1)n-2b(3b-4)$ is
\begin{align}
\gamma(b):=&\frac{1}{3}\left(b+10+\frac{b^2-10b+106}{\left(b^3+66b^2-249b+1090+9\sqrt{2b^5+40b^4-288b^3+1733b^2-2540b-36}\right)^{\frac{1}{3}}}\right.\nonumber\\
&\hspace{0.4cm}\left.+\left(b^3+66b^2-249b+1090+9\sqrt{2b^5+40b^4-288b^3+1733b^2-2540b-36}\right)^{\frac{1}{3}}\right)\nonumber
\end{align}

\begin{corollary}\label{MVStable}
The unduloids are stable, with respect to the $(n+1-b)$\textsuperscript{th} mixed-volume preserving mean curvature flow, under $(n+1-b)$\textsuperscript{th} mixed-volume preserving, axially symmetric perturbations in the following cases:
\begin{itemize}
	\item For $b=0,\ldots,3$ and $n\geq11$
	\item For $b=4,5$ and $n\geq12$
	\item For $b=6,7,8$ and $n\geq b+7$
	\item For $b\geq9$ and $n\geq b+6$.
\end{itemize}
Otherwise they are unstable.
\end{corollary}

\begin{proof}
This follows from Corollary \ref{GenStable} and we just check that (\ref{FCondMV}) is satisfied, note from its structure it is clear that the left hand side of (\ref{FCondMV}) will be negative for $n$ large enough. If $b=0,1$ then we have equality in (\ref{FCondMV}) when $n=5\pm3\sqrt{3}$, only one of which is positive and is between $10$ and $11$. For the cases of $b\geq2$ the condition is satisfied for $n>\gamma(b)$, which we evaluate for $b=2,\ldots,8$ and note that for $b\geq9$ we have $b+5<\gamma(b)<b+6$.
\end{proof}


\section{Geometric Construction}\label{SecGeoCons}
In this section we consider an alternative method for constructing the bifurcation curves of stationary solutions to the mixed-volume preserving mean curvature flow. We will use a representation of the axially symmetric CMC hypersurfaces to calculate the mixed-volume of such hypersurfaces and hence explicitly give a formula for $\eta(s)$.

The $n$-dimensional axially symmetric CMC hypersurfaces were studied in \cite{Hsiang81}, where the profile curve, $\rho(z)$, was shown to satisfy:
\begin{equation*}
z=\int_{\rho(0)}^{\rho} \frac{1}{\sqrt{\left(\frac{x^{n-1}}{C+\frac{H}{n}x^n}\right)^2-1}}\,dx,
\end{equation*}
where $C$ is a constant and $H$ is the mean curvature of the hypersurface. We note that for this representation the cylinders can only be treated through limits. Similarly, we can only treat the unduloids with half a period, i.e. when $m=1$. However, the formulas proved here can easily be extended to any amount of periods.

To obtain the hypersurfaces that satisfy the orthogonal boundary conditions we set $\left.\frac{d\rho}{dz}\right|_{z=0}=\left.\frac{d\rho}{dz}\right|_{z=d}=0$ and we will also define $s:=\frac{\rho(d)-\rho(0)}{\rho(d)+\rho(0)}$. This leads to the formulas:
\begin{equation*}
z=\rho_{0,s}\int_{1}^{\frac{\rho_s}{\rho_{0,s}}} \frac{1}{\sqrt{\left(\frac{\bar{x}^{n-1}\left((1+s)^{n}-(1-s)^{n}\right)}{2s(1+s)^{n-1}+\left((1+s)^{n-1}-(1-s)^{n-1}\right)(1-s)\bar{x}^n}\right)^2-1}}\,d\bar{x},
\end{equation*}
where we have used the change of variables $x=\rho_s(0)\bar{x}$ and have set $\rho_{0,s}:=\rho_s(0)$, which is obtained by evaluating at $z=d$ and using $\frac{\rho_s(d)}{\rho_s(0)}=\frac{1+s}{1-s}$:
\begin{equation*}
\rho_{0,s}=d\left(\int_{1}^{\frac{1+s}{1-s}} \frac{1}{\sqrt{\left(\frac{\bar{x}^{n-1}\left((1+s)^{n}-(1-s)^{n}\right)}{2s(1+s)^{n-1}+\left((1+s)^{n-1}-(1-s)^{n-1}\right)(1-s)\bar{x}^n}\right)^2-1}}\,d\bar{x}\right)^{-1}.
\end{equation*}
The mean curvature of the hypersurface is also easily obtained:
\begin{equation*}
H=\left(\frac{(1+s)^{n-1}-(1-s)^{n-1}}{(1+s)^n-(1-s)^n}\right)\frac{n(1-s)}{\rho_{0,s}}.
\end{equation*}

\begin{lemma}\label{BifCurvForm}
The bifurcation curve $\eta(s)$ is given by the formula
\begin{equation*}
\eta(s)=\frac{n-1}{d}\left(\frac{\left(\int_1^{\frac{1+s}{1-s}}g_s(x)\,dx\right)^{n+1}}{\int_1^{\frac{1+s}{1-s}}x^ng_s(x)\,dx}\right)^{\frac{1}{n}},
\end{equation*}
for the VPMCF and by
\begin{equation*}
\eta(s)=\frac{n-1}{d}\left(\frac{\left(\int_1^{\frac{1+s}{1-s}}g_s(x)\,dx\right)^{n+1-b}}{\int_1^{\frac{1+s}{1-s}}\frac{x^{n-b}g_s(x)}{\left(1+g_s(x)^{-2}\right)^{\frac{b-2}{2}}}+\frac{(b-1)x^{n+1-b}g_s'(x)}{(n+1-b)g_s(x)^2\left(1+g_s(x)^{-2}\right)^{\frac{b}{2}}}\,dx}\right)^{\frac{1}{n-b}},
\end{equation*}
for the $(n+1-b)$\textsuperscript{th} mixed-volume preserving mean curvature flow, where
\begin{equation*}
g_s(x)=\frac{1}{\sqrt{\left(\frac{x^{n-1}\left((1+s)^n-(1-s)^n\right)}{2s(1+s)^{n-1}+\left((1+s)^{n-1}-(1-s)^{n-1}\right)(1-s)x^n}\right)^2-1}}.
\end{equation*}
\end{lemma}

\begin{proof}
We first note that for the mixed-volume preserving flows $\tilde{Q}(\eta)=\binom{n}{b}\left(\frac{n-1}{\eta}\right)^{n-b}$, so that
\begin{equation}\label{QtilInv}
\tilde{Q}^{-1}(x)=(n-1)\left(\frac{\binom{n}{b}}{x}\right)^{\frac{1}{n-b}}.
\end{equation}

Now we calculate $\fint Q(u)\,dz$ for the unduloids:
\begin{equation*}
\fint_{\Sdp}Q\left(u_{\rho_{s}}\right)\,dz=\fint_0^dQ\left(\rho_{s}\right)\,dz=\left\{\begin{array}{ll} \frac{1}{d}\int_0^d \rho_{s}(z)^n\,dz & b=0\\ \frac{n}{db}\int_0^d\frac{\binom{n-1}{b-1}\rho_{s}(z)^{n-b}}{\left(1+\rho_{s}'(z)^2\right)^{\frac{b-2}{2}}} -\frac{\binom{n-1}{b-2}\rho_{s}(z)^{n+1-b}\rho_{s}''(z)}{\left(1+\rho_{s}'(z)^2\right)^{\frac{b}{2}}}\,dz & b\neq0.\end{array}\right.
\end{equation*}
Using the substitution $z(x)=\rho_{0,s}\int_1^xg_s(y)\,dy$, we have that $\rho_{s}\left(z(x)\right)=\rho_{0,s}y$, $\frac{dz}{dx}=\rho_{0,s}g_s(x)$ and
\begin{equation*}
\fint_{\Sdp}Q\left(u_{\rho_{s}}\right)\,dz=\left\{\begin{array}{ll} \frac{\rho_{0,s}^{n+1}}{d}\int_1^{\frac{1+s}{1-s}}x^ng_s(x)\,dx & b=0\\ \frac{\binom{n}{b}\rho_{0,s}^{n+1-b}}{d}\int_1^{\frac{1+s}{1-s}}\frac{x^{n-b}}{\left(1+g_s(x)^{-2}\right)^{\frac{b-2}{2}}} +\frac{(b-1)x^{n+1-b}g_s'(x)}{(n+1-b)\left(1+g_s(x)^{-2}\right)^{\frac{b}{2}}g_s(x)^3}\,dx & b\neq0.\end{array}\right.
\end{equation*}
From the second point of Lemma \ref{defpsi} $\fint_{\Sdp}Q\left(u_{\rho_{s}}\right)\,dz=\tilde{Q}(\eta_1(s))$, so by using (\ref{QtilInv}) and the fact that $\rho_{0,s}=d\left(\int_0^{\frac{1+s}{1-s}}g_s(x)\,dx\right)^{-1}$ we obtain the result.
\end{proof}

\begin{figure}[ht]
        \centering
        \begin{subfigure}[b]{0.3\textwidth}
                \centering
                \includegraphics[width=\textwidth]{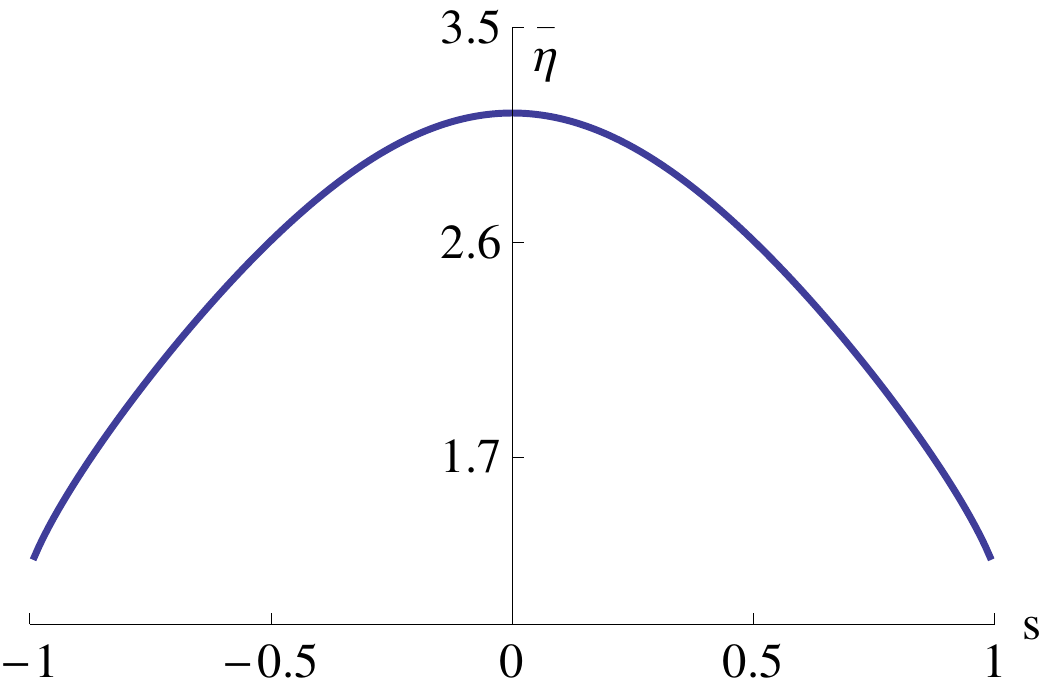}
                \caption{$\bar{\eta}(s)$ in dimension $n=2$}
                \label{fig:Vol2}
        \end{subfigure}%
        ~ 
        \begin{subfigure}[b]{0.3\textwidth}
                \centering
                \includegraphics[width=\textwidth]{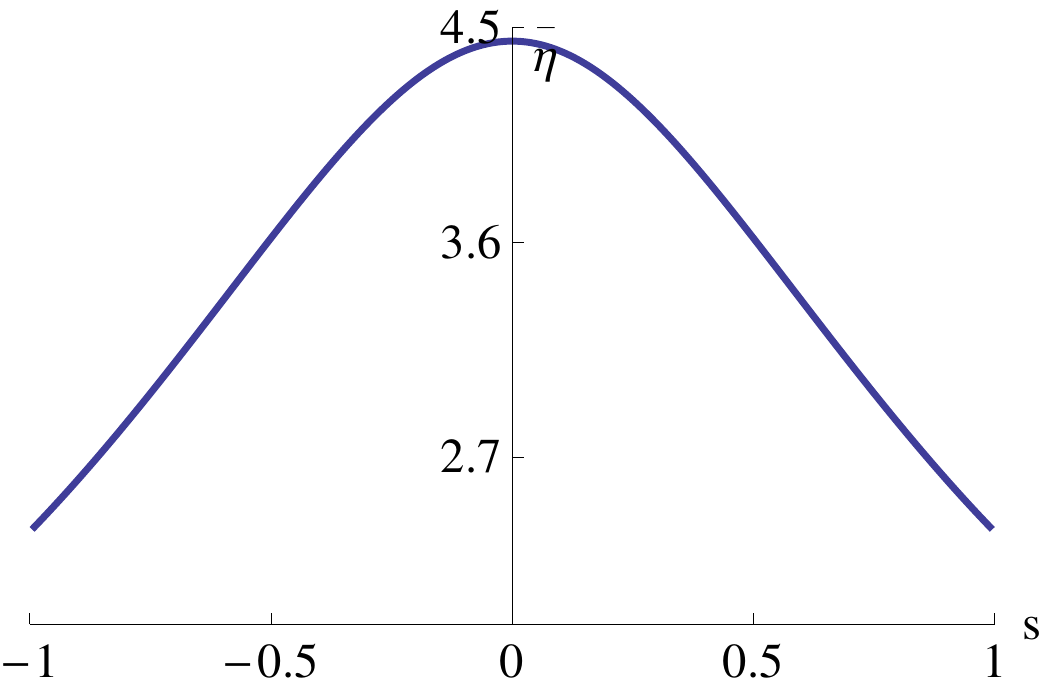}
                \caption{$\bar{\eta}(s)$ in dimension $n=3$}
                \label{fig:Vol3}
        \end{subfigure}
        ~ 
        \begin{subfigure}[b]{0.3\textwidth}
                \centering
                \includegraphics[width=\textwidth]{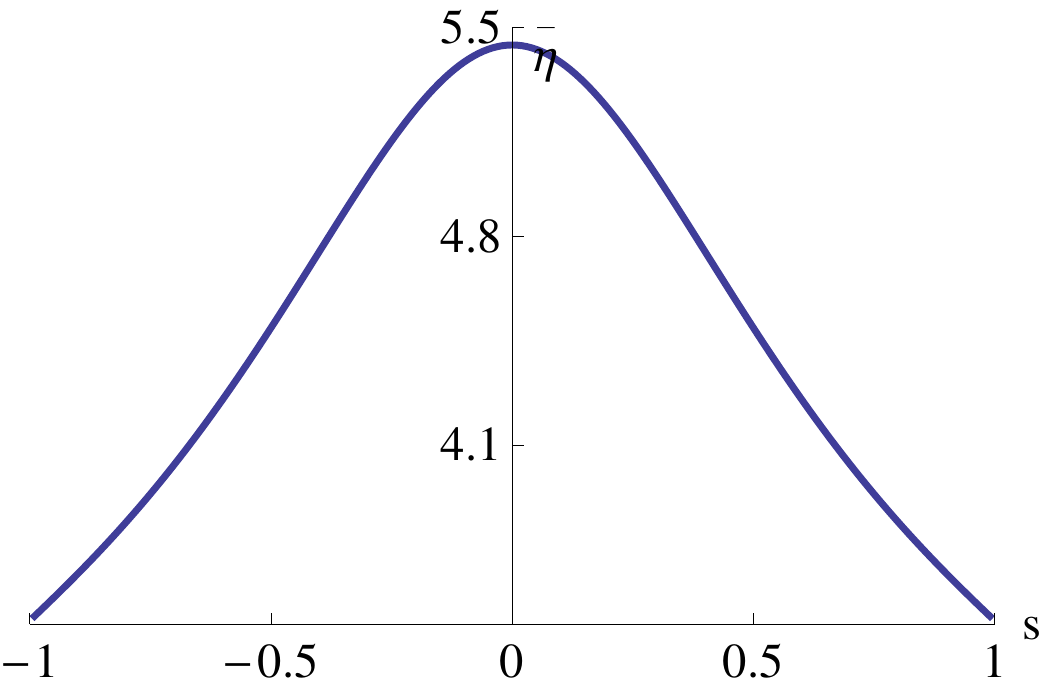}
                \caption{$\bar{\eta}(s)$ in dimension $n=4$}
                \label{fig:Vol4}
        \end{subfigure}
      \\
        \begin{subfigure}[b]{0.3\textwidth}
                \centering
                \includegraphics[width=\textwidth]{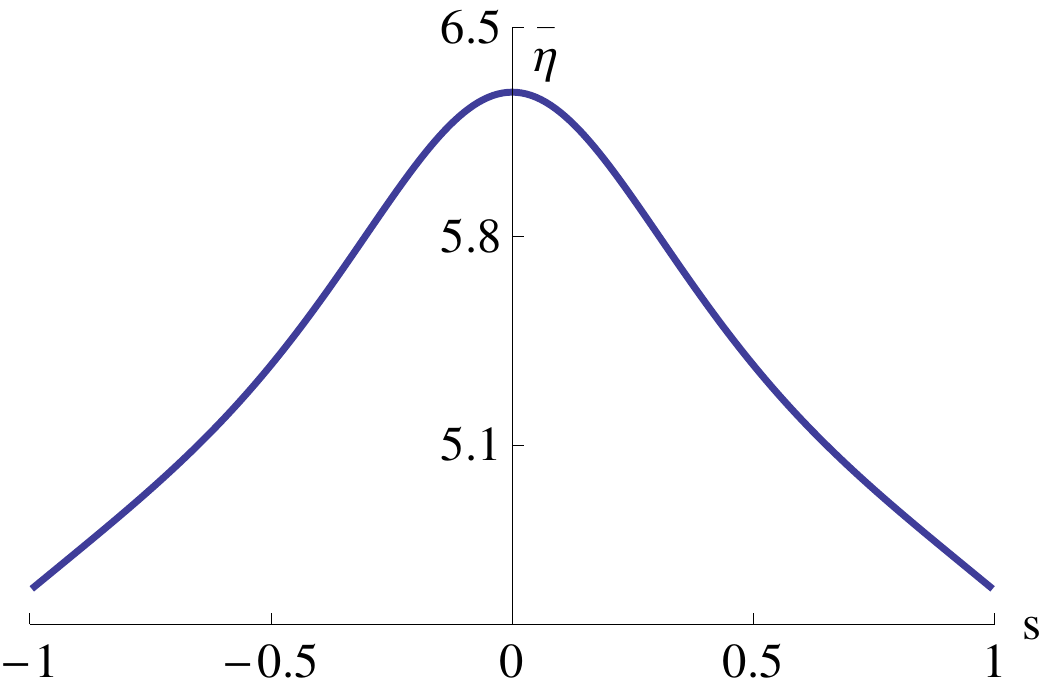}
                \caption{$\bar{\eta}(s)$ in dimension $n=5$}
                \label{fig:Vol5}
        \end{subfigure}%
        ~ 
        \begin{subfigure}[b]{0.3\textwidth}
                \centering
                \includegraphics[width=\textwidth]{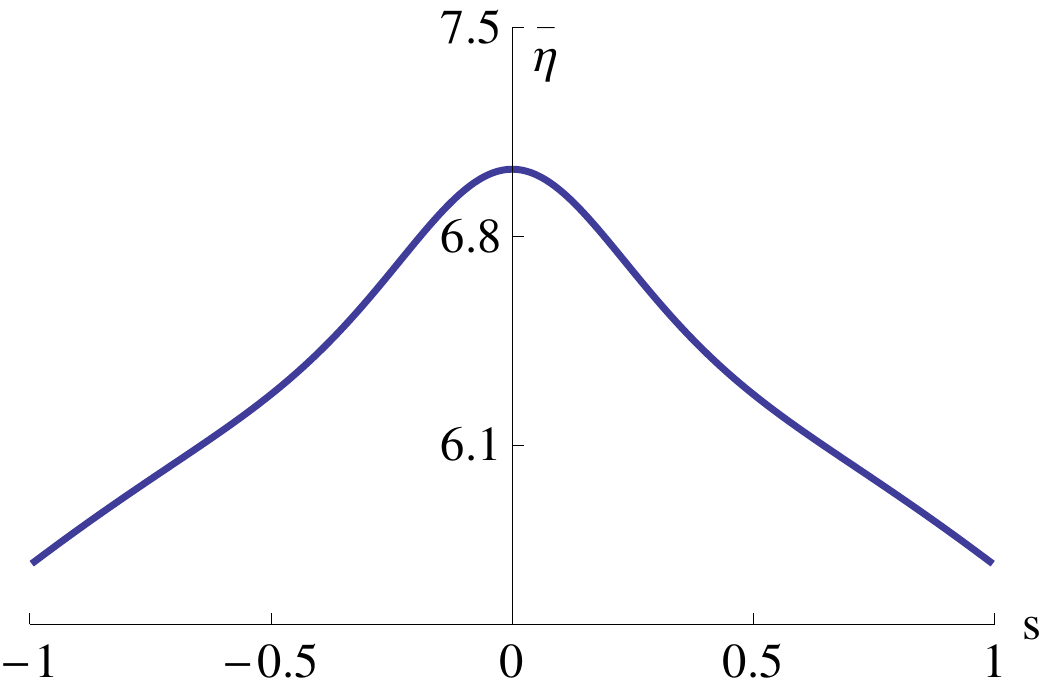}
                \caption{$\bar{\eta}(s)$ in dimension $n=6$}
                \label{fig:Vol6}
        \end{subfigure}
        ~ 
        \begin{subfigure}[b]{0.3\textwidth}
                \centering
                \includegraphics[width=\textwidth]{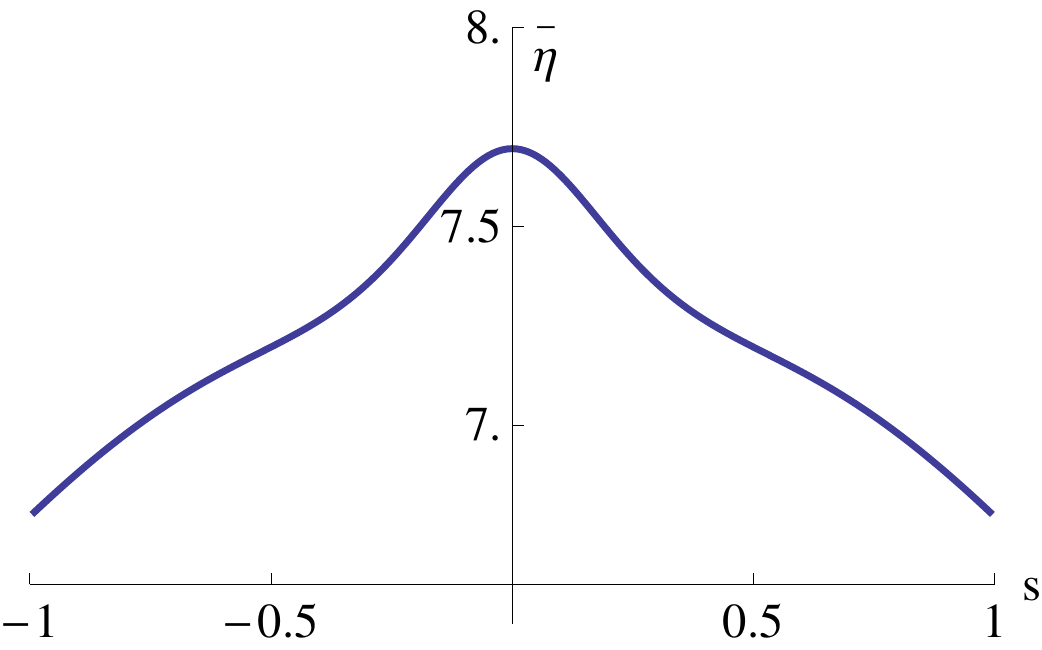}
                \caption{$\bar{\eta}(s)$ in dimension $n=7$}
                \label{fig:Vol7}
        \end{subfigure}
      \\
        \begin{subfigure}[b]{0.3\textwidth}
                \centering
                \includegraphics[width=\textwidth]{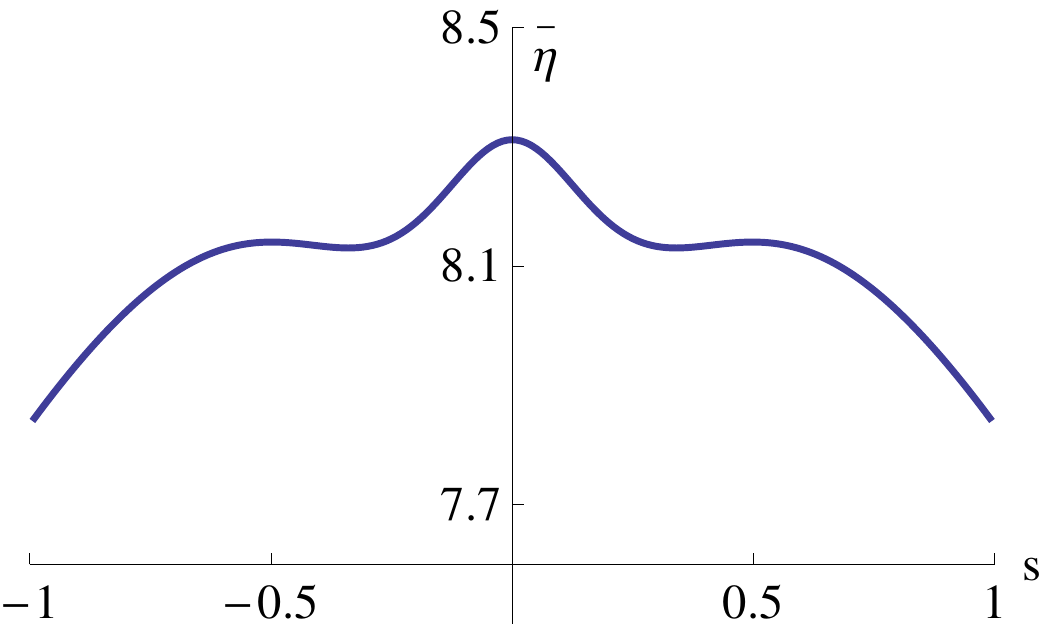}
                \caption{$\bar{\eta}(s)$ in dimension $n=8$}
                \label{fig:Vol8}
        \end{subfigure}%
        ~ 
        \begin{subfigure}[b]{0.3\textwidth}
                \centering
                \includegraphics[width=\textwidth]{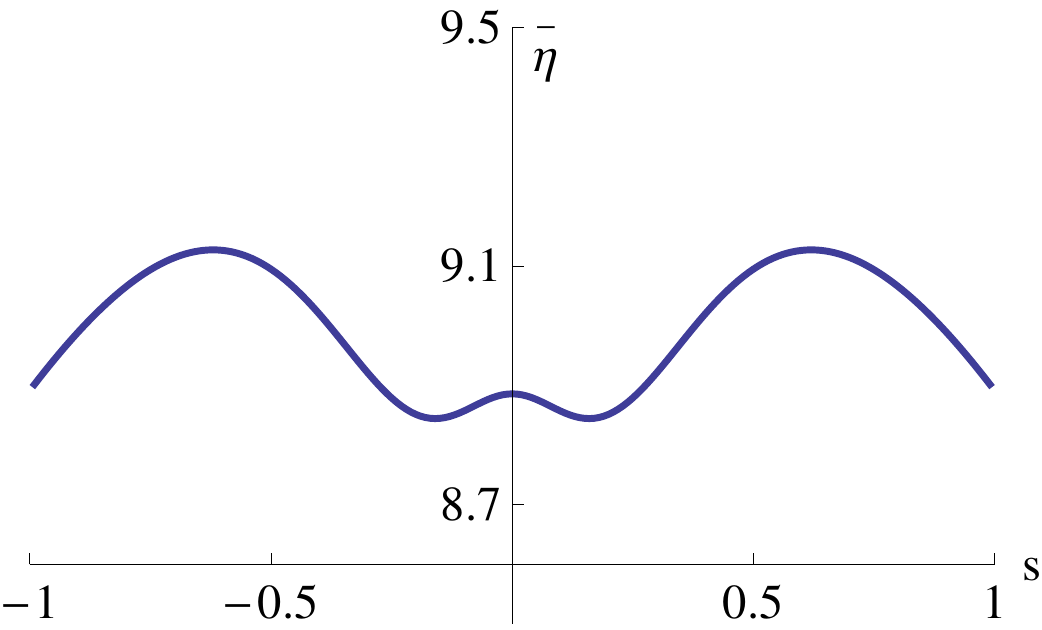}
                \caption{$\bar{\eta}(s)$ in dimension $n=9$}
                \label{fig:Vol9}
        \end{subfigure}
        ~ 
        \begin{subfigure}[b]{0.3\textwidth}
                \centering
                \includegraphics[width=\textwidth]{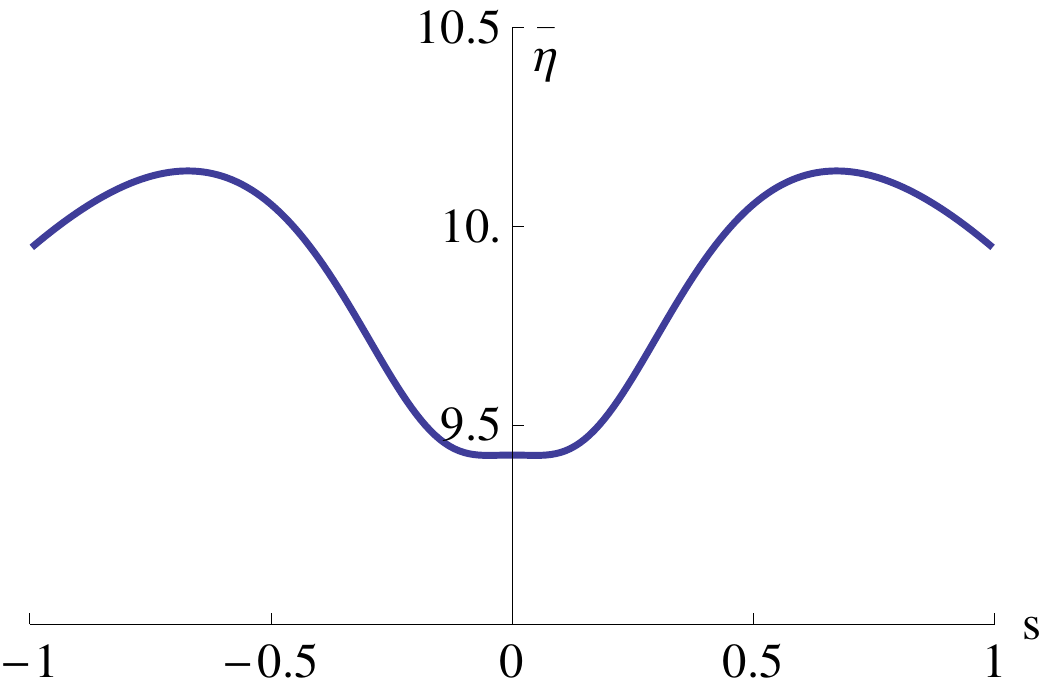}
                \caption{$\bar{\eta}(s)$ in dimension $n=10$}
                \label{fig:Vol10}
        \end{subfigure}
      \\
        \begin{subfigure}[b]{0.3\textwidth}
                \centering
                \includegraphics[width=\textwidth]{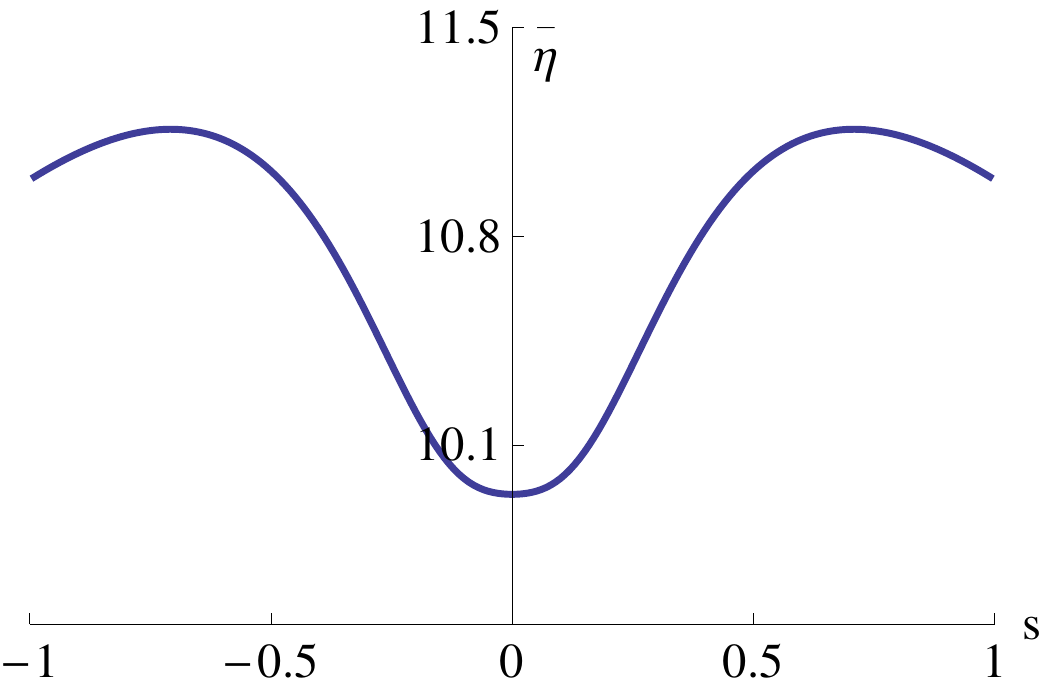}
                \caption{$\bar{\eta}(s)$ in dimension $n=11$}
                \label{fig:Vol11}
        \end{subfigure}
        ~ 
        \begin{subfigure}[b]{0.3\textwidth}
                \centering
                \includegraphics[width=\textwidth]{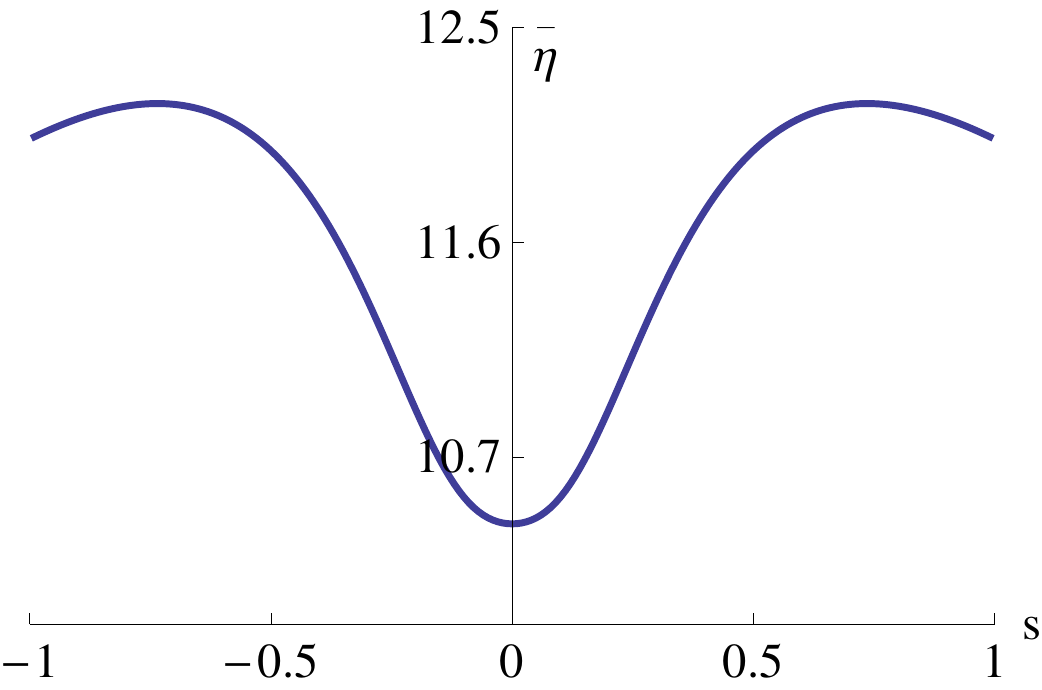}
                \caption{$\bar{\eta}(s)$ in dimension $n=12$}
                \label{fig:Vol12}
        \end{subfigure}
        \begin{subfigure}[b]{0.3\textwidth}
                \centering
                \includegraphics[width=\textwidth]{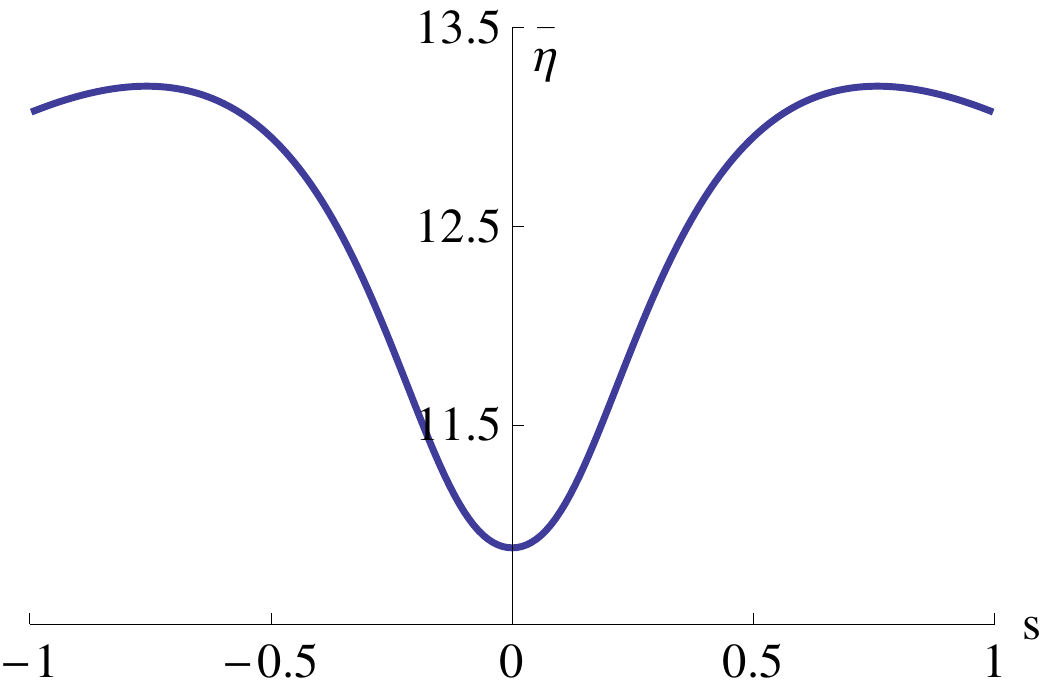}
                \caption{$\bar{\eta}(s)$ in dimension $n=13$}
                \label{fig:Vol13}
        \end{subfigure}
        \caption{Normalised bifurcation parameter in different dimensions}\label{fig:UnduloidEtaBar}
\end{figure}

\begin{figure}[ht]
        \centering
        \begin{subfigure}[b]{0.45\textwidth}
                \centering
                \includegraphics[width=0.82\textwidth]{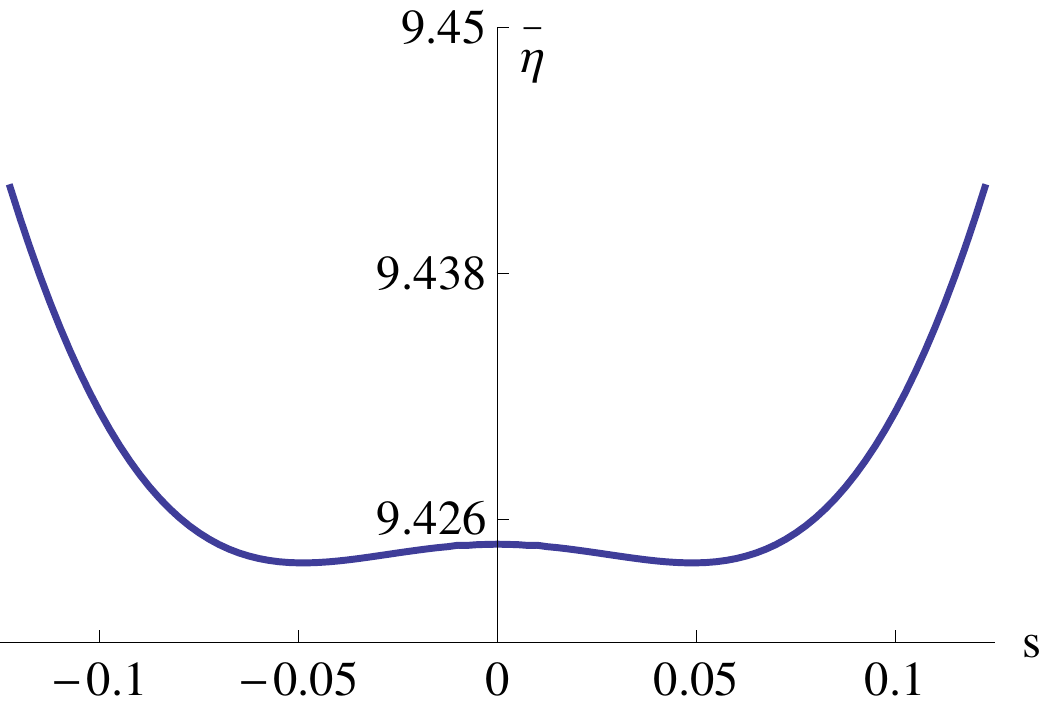}
                \caption{Close up of $\bar{\eta}(s)$ in dimension $n=10$}
                \label{fig:Vol10zoom}
        \end{subfigure}%
        ~ 
        \begin{subfigure}[b]{0.45\textwidth}
                \centering
                \includegraphics[width=0.82\textwidth]{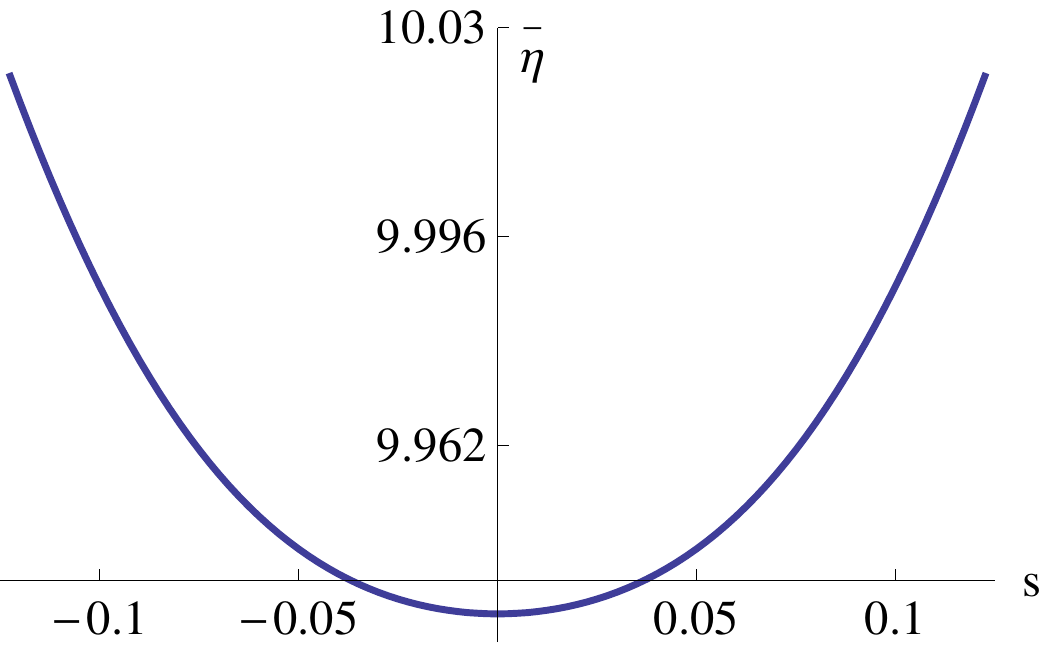}
                \caption{Close up of $\bar{\eta}(s)$ in dimension $n=11$}
                \label{fig:Vol11zoom}
        \end{subfigure}
        \caption{Turning point of the normalised bifurcation parameter for $n=10,11$}\label{fig:UnduloidEtaBarZoom}
\end{figure}

Figure \ref{fig:UnduloidEtaBar} shows these bifurcation curves for the case of the VPMCF and various $n$ values. These plots confirm that the bifurcation parameter (volume enclosed) is a local maximum (minimum) at the cylinder if $n\leq10$, while for $n\geq11$ it is a local minimum (maximum) at the cylinder; see Figure \ref{fig:UnduloidEtaBarZoom} for a close up of the turning point for dimensions ten and eleven. Interesting phenomena are also apparent in dimensions eight and higher where additional turning points appear. In dimension eight, a local maximum and minimum of the enclosed volume occur within the family of unduloids. In dimensions nine and ten, the turning points separate from each other and these points are the global maximum and minimum volume of the family. In dimensions eleven and higher only the local minimum of the volume occurs and it remains a global minimum volume of the family. This behaviour is very intriguing and it would be of interest to know what is special about these unduloids.


\appendix


\section{A Flow Invariant}\label{AppInv}
In this appendix we aim to determine under what conditions the flow (\ref{WVPCF}) has an invariant weighted-volume type quantity. For simplicity we consider the case where the hypersurfaces are axially symmetric, i.e. equation (\ref{WVPCFGraph}), in which case the question becomes when does there exist a second order operator $Q$ such that $\int Q(\rho)\,dz$ is independent of time. To perform the analysis we again consider the equivalent flow on the circle given in (\ref{WVPCFCircle}). To simplify notation we will define $l(u'):=\sqrt{1+u'^2}$.

\begin{lemma}
There exists $Q(u)=q(u,u',u'')$ such that
\begin{equation}\label{QIdent}
\int_{\Sdp}DQ(u)[v]\,dz=n\int_{\Sdp}v\Xi\left(\kapsurf{u}\right)u^{n-1}\,dz
\end{equation}
if and only if $\Xi\left(\bm{\kappa}\right)=\sum_{a=0}^n c_aE_a\left(\bm{\kappa}\right)$ for some constants $c_a\in\Rone$. In this case (\ref{QIdent}) is satisfied by the $Q$ in (\ref{defQ}).
\end{lemma}

\begin{proof}
We proceed with the `only if' case by calculating the linearisation of $Q$ inside an integral and use integration by parts
\begin{align*}
\int_{\Sdp}DQ(u)[v]\,dz=& \int_{\Sdp}\frac{\partial q}{\partial x_1}\left(u,u',u''\right)v+\frac{\partial q}{\partial x_2}\left(u,u',u''\right)\frac{\partial v}{\partial z} +\frac{\partial q}{\partial x_3}\left(u,u',u''\right)\frac{\partial^2 v}{\partial z^2}\,dz\\
=& \int_{\Sdp}v\left(\frac{\partial q}{\partial x_1}(u,u',u'')-\frac{\partial}{\partial z}\left(\frac{\partial q}{\partial x_2}\left(u,u',u''\right) -\frac{\partial}{\partial z}\left(\frac{\partial q}{\partial x_3}\left(u,u',u''\right)\right)\right)\right)\,dz.
\end{align*}
By now expanding the derivatives we obtain
\begin{align*}
\int_{\Sdp}DQ(u)[v]\,dz=& \int_{\Sdp}\left(\frac{\partial q}{\partial x_1}(u,u',u'')-\frac{\partial^2q}{\partial x_1\partial x_2}\left(u,u',u''\right)u'-\frac{\partial^2q}{\partial x_2^2}\left(u,u',u''\right)u''\right.\\
&\hspace{0.8cm}\left. -\frac{\partial^2q}{\partial x_3\partial x_2}\left(u,u',u''\right)u''' +\frac{\partial^3 q}{\partial x_1^2\partial x_3}\left(u,u',u''\right)u'^2\right.\\
&\hspace{0.8cm}\left.  +2\frac{\partial^3q}{\partial x_1\partial x_2\partial x_3}\left(u,u',u''\right)u'u'' +2\frac{\partial^3q}{\partial x_1\partial x_3^2}\left(u,u',u''\right)u'u'''\right.\\
&\hspace{0.8cm}\left. +\frac{\partial^2q}{\partial x_1\partial x_3}\left(u,u',u''\right)u'' +\frac{\partial^3q}{\partial x_2^2\partial x_3}\left(u,u',u''\right)u''^2\right.\\
&\hspace{0.8cm}\left. +2\frac{\partial^3 q}{\partial x_2\partial x_3^2}\left(u,u',u''\right)u''u'''+\frac{\partial^2q}{\partial x_2\partial x_3}\left(u,u',u''\right)u'''\right.\\
&\hspace{0.8cm}\left. +\frac{\partial^3q}{\partial x_3^3}\left(u,u',u''\right)u'''^2 +\frac{\partial^2q}{\partial x_3^2}\left(u,u',u''\right)u''''\right)v\,dz.
\end{align*}
However since $\Xi\left(\kapsurf{u}\right)u^{n-1}$ does not depend on $u''''$ we see that we require $\frac{\partial^2q}{\partial x_3^2}=0$ and hence $q$ is linear in $x_3$. Let $q\left(x_1,x_2,x_3\right)=\alpha\left(x_1,x_2\right)x_3+\beta\left(x_1,x_2\right)$, then:
\begin{align*}
\int_{\Sdp}DQ(u)[v]\,dz=& \int_{\Sdp}\left(2\frac{\partial\alpha}{\partial x_1}\left(u,u'\right)u'' +\frac{\partial\beta}{\partial x_1}\left(u,u'\right) +\frac{\partial^2\alpha}{\partial x_1\partial x_2}\left(u,u'\right)u'u'' -\frac{\partial^2\beta}{\partial x_1\partial x_2}\left(u,u'\right)u'\right.\\
&\hspace{0.8cm}\left. -\frac{\partial^2\beta}{\partial x_2^2}\left(u,u'\right)u'' +\frac{\partial^2\alpha}{\partial x_1^2}\left(u,u'\right)u'^2\right)v\,dz.
\end{align*}
This is now linear in $u''$, hence $\Xi\left(\kapsurf{u}\right)u^{n-1}$ must also be linear in $u''$ for (\ref{QIdent}) to hold. This in turn means that $\Xi\left(\bm{\kappa}\right)$ must be linear in $\kappa_n$, however by symmetry it is therefore linear in all $\kappa_a$. Thus it must be a linear combination of the $E_a$s.

To see that if $\Xi$ is a linear combination of the $E_a$s we do obtain a $Q$ such that (\ref{QIdent}) holds we linearise the $Q$ given in (\ref{defQ}) in parts. Firstly consider $Q_0(u)=\frac{1}{n}u^n$:
\begin{equation*}
\int_{\Sdp}DQ_0(u)[v]\,dz=\int_{\Sdp}u^{n-1}v\,dz=\int_{\Sdp}vE_0u^{n-1}\,dz.
\end{equation*}

Now consider $Q_a(u)=\frac{1}{a}E_{a-1}\left(\kapsurf{u}\right)u^{n-1}L(u)$ for $1\leq a\leq n+1$ and use the notation $\binom{b}{k}=0$ for $k>b$. This means that
\begin{equation*}
q\left(x_1,x_2,x_3\right)=\frac{1}{a}\left(\binom{n-1}{a-1}x_1^{n-a}l(x_2)^{2-a}-\binom{n-1}{a-2}x_1^{n+1-a}l(x_2)^{-a}x_3\right),
\end{equation*}
with $\frac{dl}{dx_2}=x_2l(x_2)^{-1}$. A standard computation then gives
\begin{align*}
\frac{\partial q}{\partial x_1}\left(u,u',u''\right)-\frac{\partial}{\partial z}\left(\frac{\partial q}{\partial x_2}\left(u,u',u''\right)-\frac{\partial}{\partial z}\left(\frac{\partial q}{\partial x_3}\left(u,u',u''\right)\right)\right)&\\
&\hspace{-3cm}=\binom{n-1}{a}u^{n-a-1}l(u')^{-a}-\binom{n-1}{a-1}u^{n-a}l(u')^{-(a+2)}u''\\
&\hspace{-3cm}=E_{a}\left(\kapsurf{u}\right)u^{n-1},
\end{align*}
where we set $E_{n+1}=0$. Using the formula for $\int_{\Sdp}DQ(u)[v]\,dz$ from the start of this proof we have:
\begin{equation*}
\int_{\Sdp}DQ_a(u)[v]\,dz=\int_{\Sdp}vE_{a}\left(\kapsurf{u}\right)u^{n-1},
\end{equation*}

Therefore if $\Xi\left(\bm{\kappa}\right)=\sum_{a=0}^n c_aE_a\left(\bm{\kappa}\right)$ and we set $Q(u)=n\sum_{a=0}^{n+1} c_aQ_a(u)$ for some $c_{n+1}\in\Rone$ (as in equation (\ref{defQ})) then we have
\begin{align*}
\int_{\Sdp}DQ(u)[v]\,dz=&n\sum_{a=0}^{n+1}c_a\int_{\Sdp}DQ_a(u)[v]\,dz\\
=&n\sum_{a=0}^{n+1} c_a\int_{\Sdp}vE_{a}\left(\kapsurf{u}\right)u^{n-1}\,dz\\
=&n\int_{\Sdp}v\left(\sum_{a=0}^n c_aE_{a}\left(\kapsurf{u}\right)\right)u^{n-1}\,dz\\
=&n\int_{\Sdp}v\Xi\left(\kapsurf{u}\right)u^{n-1}\,dz.
\end{align*}
\end{proof}

The conditions for an invariant of the flow follow easily.
\begin{corollary}\label{CorInvar}
There exists a non-zero invariant of the flow (\ref{WVPCFCircle}) of the form $\int_{\Sdp}Q(u)\,dz$, where $Q$ is a second order operator satisfying $Q(0)=0$, if and only if $\Xi\left(\bm{\kappa}\right)=\sum_{a=0}^n c_aE_a\left(\bm{\kappa}\right)$ for some $c_a\in\Rone$, $0\leq a\leq n$.
\end{corollary}

\begin{proof}
Set $Q$ as in (\ref{defQ}). Then
\begin{align*}
\frac{d}{dt}\left(\int_{\Sdp}Q(u)\,dz\right) =& \int_{\Sdp}DQ(u)\left[\frac{\partial u}{\partial t}\right]\,dz\\
=&n\int_{\Sdp}\frac{\partial u}{\partial t}\Xi\left(\kapsurf{u}\right)u^{n-1}\,dz\\
=&n\int_{\Sdp}\left(\frac{1}{\int_{\Sdp}\Xi\left(\kapsurf{u}\right)\,d\musurf{u}}\int_{\Sdp}F\left(\kapsurf{u}\right)\Xi\left(\kapsurf{u}\right)\,d\musurf{u} -F\left(\kapsurf{u}\right)\right)\Xi\left(\kapsurf{u}\right)\,d\musurf{u}\\
=&0.
\end{align*}
\end{proof}

Using the theorem by Hadwiger \cite{Hadwiger57} we also obtain the following corollary.
\begin{corollary}
Any continuous, rigid motion invariant valuation is an invariant for a flow of the form (\ref{WVPCF}).
\end{corollary}


\bibliographystyle{plain}
\bibliography{MyBibliography}


\end{document}